\documentclass[preprint,10pt]{elsarticle}
\biboptions{numbers,sort&compress}
\usepackage{amsfonts,psfrag,amsmath,mathrsfs,amsthm,color,empheq}
\usepackage{color,amssymb,bm}
\usepackage{epsfig,soul}
\usepackage{hyperref}
\hypersetup{
	colorlinks   = true, %Colours links instead of ugly boxes
	urlcolor     = blue, %Colour for external hyperlinks
	linkcolor    = blue, %Colour of internal links
	citecolor   = red %Colour of citations
}
\usepackage{multirow}

 \renewenvironment{thebibliography}[1]{%
 	\begin{oldthebibliography}{#1}%
 		\setlength{\parskip}{0.1ex}%
 		\setlength{\itemsep}{0.1ex}%
 	}%
 	{%
 	\end{oldthebibliography}%
 }

\usepackage{amsmath,mathtools} % assumes amsmath package installed
\usepackage{amsthm,amssymb,bbm}
 \usepackage{graphicx,float}
 \usepackage{newtxmath}
\usepackage{helvet}

\usepackage[T1]{fontenc}
\DeclareMathAlphabet{\mathcal}{OMS}{cmsy}{m}{n}

\usepackage{setspace}
\AtBeginDocument{%
	\addtolength\abovedisplayskip{-0.2\baselineskip}%
	\addtolength\belowdisplayskip{-0.2\baselineskip}%
}
 \usepackage[margin=1in]{geometry}

\newcommand{\tn}[1]{\quad \textnormal{#1}\quad }

\newcommand{\IR}{{\mathbb{R}}}

\newcommand{\CO}{{\mathcal{O}}}

\newcommand{\T}{{\intercal}}
\newcommand{\diag}{{\rm{diag}}}

\newcommand{\bmt}{\left[ \begin{array}{ccccccccc}}
\newcommand{\emt}{\end{array}\right]}
\newcommand{\bean}{\begin{eqnarray*}}
\newcommand{\eean}{\end{eqnarray*}}
\newcommand{\bea}{\begin{eqnarray}}
\newcommand{\eea}{\end{eqnarray}}
\newcommand{\eq}{\begin{equation}\begin{array}{lllllllll}}
\newcommand{\ee}{\end{array}\end{equation}}
\newcommand{\eqn}{\begin{equation*}\begin{array}{lllllllll}}
\newcommand{\een}{\end{array}\end{equation*}}
\newtheorem{theorem}{Theorem}[section]
\newtheorem{lemma}{Lemma}[section]
\newtheorem{remark}{Remark}[section]

\allowdisplaybreaks

\graphicspath{{./Figures/}}

\journal{Inverse Problems}

\begin{document}

	\begin{frontmatter}
	
	\title{A Direct Parallel-in-Time Quasi-Boundary Value Method for\\Inverse Space-Dependent Source Problems}
	
	 	\author{Yi Jiang}
	 \ead{yjianaa@siue.edu}  	
	 \address{Department of Mathematics and Statistics, Southern Illinois University Edwardsville, Edwardsville, IL 62026, USA.}

	\author{Jun Liu\corref{mycorrespondingauthor}}
	\cortext[mycorrespondingauthor]{Corresponding author}
	\ead{juliu@siue.edu}  	
	\address{Department of Mathematics and Statistics, Southern Illinois University Edwardsville, Edwardsville, IL 62026, USA.}

	\author{Xiang-Sheng Wang}
	\ead{xswang@louisiana.edu}  	
	\address{Department of Mathematics, University of Louisiana at Lafayette, Lafayette, LA 70503, USA.}
	
	\begin{abstract}
Inverse source problems arise often in real-world applications, such as localizing unknown groundwater contaminant sources.
Being different from Tikhonov regularization, the quasi-boundary value method has been proposed and analyzed as an effective way for regularizing such inverse source problems, which was shown to achieve an optimal order convergence rate under suitable assumptions.
However,  fast direct or iterative solvers for the resulting all-at-once large-scale linear systems have been rarely studied in the literature. In this work, we first proposed and analyzed a modified quasi-boundary value method, and then developed a  diagonalization-based parallel-in-time (PinT) direct solver, which can achieve a dramatic speedup in CPU times when compared with MATLAB's sparse direct solver.	
In particular, the time-discretization matrix $B$ is shown to be diagonalizable, and the condition number of its eigenvector matrix $V$ is proven to exhibit quadratic growth, which guarantees the roundoff errors due to diagonalization is well controlled.
Several 1D and 2D examples are presented to demonstrate the very promising computational efficiency of our proposed method,
where the CPU times in 2D cases can be speedup by three orders of magnitude.
	\end{abstract}
	
	\begin{keyword}
	ill-posed \sep inverse source problem \sep quasi-boundary value method\sep regularization \sep diagonalization \sep parallel-in-time\sep condition number
	\end{keyword}
	
\end{frontmatter}

%%%%%%%%%%%%%%%%%%%%%%%%%%%%%%%%%%%%%%%%%%%%%%%%%%%%%%%%%%%%%%%%%%%%%%%%%%%%%%%%
\section{Introduction}
\label{SecProblem}
 Let $T>0$ and $\Omega\subset \IR^d  (d=1,2,3)$ be an open and bounded domain with a piecewise smooth boundary $\partial\Omega$.
 We consider the inverse source problem (ISP) \cite{SAVATEEV_1995,Cannon_1998,dou2009identifying} of reconstructing the unknown space-dependent source term $f\in L^2(\Omega)$ from the final time condition $g=u(\cdot,T)  \in H_0^1(\Omega)$, according to a non-homogeneous heat equation
 \eq \label{state}
 \left\{\begin{array}{ll}
 	u_t -\Delta u =f,\ \qquad &\tn{in} \Omega\times (0,T),  \\
 	u(\cdot,t)=0, &\tn{on} \partial\Omega\times (0,T), \\
 	u(\cdot,0)=\phi , &\tn{in} \Omega, \\
 	u(\cdot,T)=g , &\tn{in} \Omega,
 \end{array}\right.
 \ee
 where $\phi\in H_0^1(\Omega)$ is a given initial condition.
 In practice, the final condition $g$ is unknown exactly and it is available as a noisy measurement $g_\delta\in  L^2(\Omega)$, which is assumed to  satisfy $\|g-g_\delta\|_{2}\le \delta$ with a given noise level $\delta>0$.
  This leads to an  ill-posed inverse problem that requires effective regularization techniques for stable numerical approximations \cite{engl2000,kabanikhin2011inverse,kirsch2021introduction,lesnic2021inverse}.

Many research has been dedicated to the inverse source problem since 1970s, where the desired source term is usually assumed to have \textit{a priori} form. For $f$ that depends on the state function $u$, the problem was investigated in \cite{ Cannon_1998, fatullayev02, fatullayev04}.
For $f$ that depends on space or time variable only, many regularization methods have been developed such as Fourier method \cite{dou2009optimal}, quasi-reversibility method \cite{dou2009identifying}, quasi-boundary value method \cite{Yang_2013} and simplified Tikhonov regularization method \cite{yang2010simplified}. In particular, in \cite{dou2009identifying, yang2010simplified, Yang_2013}, the original problem is perturbed by a regularization parameter and the unknown source term is expressed in the form of a series expansion tailored by a regularizing filter. Following an appropriate convergence analysis, the regularization parameter in these work is determined to balance the approximation accuracy and the stability of the regularized problem. The Fourier method in \cite{dou2009optimal} solves the problem in the frequency domain and alleviates the ill-posedness of the problem by cutting off the high frequency components in the source solution, where the cut-off frequency is also chosen based on a convergence analysis. Such idea that truncates the terms contributing to the ill-posedness is also seen in \cite{yanfudou10}. In this work, a finite difference method is used to solve the inverse problem and the resulting linear system is solved by the singular value decomposition, where the small singular values are filtered out based on generalized cross-validation criterion \cite{golubheathwahba79}. There are some other numerical methods adopted in the research on inverse source problem, usually in conjunction with a classical regularization technique like Tikhonov method. For example, the boundary element method \cite{farcas2006boundary}, the method of fundamental solutions \cite{yan2008method, ahmadabadi2009method, yan2009meshless} and the finite element method \cite{wangzhangwu16}. Some iterative algorithms can be found in \cite{johansson2007variational, johansson2007determination, yangdehghanyuluo11, yangyuluodeng13}. For $f$ that is a function of both time and space variables but is additive or separable, we refer to \cite{yimurio04, tronglongalain05, trongquanalain06,  mafuzhang12}.

%and the conjugate gradient method \cite{johansson2007variational}

 % variational method\cite{johansson2007variational}, boundary-element method\cite{farcas2006boundary}, fixed-point iterative method \cite{johansson2007determination}, method of fundamental solutions \cite{yan2008method,ahmadabadi2009method}, Fourier regularization method \cite{dou2009optimal}, meshless method \cite{yan2009meshless}, quasi-reversibility method \cite{dou2009identifying}, quasi-boundary value method \cite{Yang_2013},  simplified Tikhonov regularization method \cite{yang2010simplified},

  %\hl{Review regularization methods for time-dependent ISP:}

  %\hl{Brief review for time-fractional ISP (not our focus):}

  For solving direct (or forward) evolutionary PDEs, many efficient parallelizable numerical algorithms  have been developed in the last few decades due to the advent of massively parallel computers. In addition to the achieved high parallelism in space,  a lot of recent advances in various parallel-in-time (PinT) algorithms for solving forward time-dependent PDE problems were reviewed in \cite{gander201550}.
  However, the application of such PinT algorithms to ill-posed inverse PDE problems were rarely investigated in the literature, except in a short paper \cite{daoud2007stability} about the \textit{parareal} algorithm for a different parabolic inverse problem and another earlier paper \cite{lee2006parallel} based on numerical (inverse) Laplace transform techniques in time.
  One obvious difficulty is how to address the underlying regularization treatment in the framework of PinT algorithms, which seems to be highly dependent on the regularized problem structure and discretization schemes.
  Inspired by several recent works \cite{MR08,MPW18,GH19,WuLiu2020,LW20,LiuWang2022} on diagonalization-based PinT algorithms,  we propose to redesign the existing quasi-boundary value methods in a structured manner such that the diagonalization-based PinT direct solver can be successfully employed. Such a PinT direct solver can greatly speed up the quasi-boundary value methods while achieving a comparable reconstruction accuracy. Recently, such an interesting approach of integrating  PinT direct solver with regularization was  applied to backward heat conduction problems \cite{Seidman1996,tautenhahn1996optimal,liu2019quasi}, where a  block $\omega$-circulant structure was exploited for developing a fast FFT-based direct PinT solver \cite{liu2021fast}.

 Besides the above mentioned ISPs for PDEs based on ordinary integer-order derivatives, there are several recent works on solving ISPs in the framework of time-fractional PDEs, to name just a few \cite{jin2015tutorial,Wei_2014a,Wei_2014b,yang2015inverse,nguyen2016regularized,wei2016inverse,ali2020inverse}.
 The majority of these contributions focuses on the convergence properties of the proposed regularization techniques without discussing fast algorithms for their numerical implementations.
 For solving related time-fractional diffusion inverse source problems, the authors in \cite{Ke2020} proposed a fast structured preconditioner based on approximate Schur complement and block $\omega$-circulant matrix. However, as an iterative solver, the underlying convergence analysis  for preconditioned GMRES in \cite{Ke2020} is a daunting task (due to nonsymmetric systems) and the preconditioner can not be easily parallelized in time. Our proposed direct PinT solver does not have such limitations for its practical use.
 Due to very different discretization schemes in time, we mention that our proposed direct PinT solver may not be directly applicable to such ISPs with  time-fractional PDEs. Nevertheless, we believe the similar algorithm can be used upon modification.

  In this paper we designed and analyzed a new parameterized quasi-boundary value method (PQBVM) for regularizing the ISPs, where a well-conditioned diagonalization-based PinT direct solver is developed for its efficient numerical implementation. The major goal is to improve the overall computational efficiency in terms of CPU times, without obviously degrading the convergence rates in comparison with existing methods.
  As theoretical contributions,  the condition number of the diagonalization of the time discretization matrix is rigorously estimated and the convergence rate  of the new PQBVM is also shown with suitable a  priori choice of the regularization parameter. For 2D problems with a small $64^3$ mesh (see results in Table \ref{T4A}), our direct PinT  solver can drastically speed up the CPU times of the standard QBVM based on sparse direct solver from  over 2 mins to about 0.04 second (on a desktop PC).

  	The rest of this paper is organized as follows.
  	In the next Section 2, we propose a new parameterized QBVM based on finite difference discretization and present a diagonalization-based direct PinT solver based on the derived system structure. Section 3 is devoted to justifying when the time discretization matrix $B=VDV^{-1}$ is indeed diagonalizable and, more importantly, estimating the growth rate of the condition number of its eigenvector matrix with a special choice of the free parameter.
  	The convergence analysis with suitable choice of the regularization parameter is given in Section 4.
  	Several numerical examples are presented to illustrate the high efficiency of the proposed algorithm in Section 5. Finally, some conclusions are made in Section 6.
 \section{A new quasi-boundary value method and its PinT implementation}
 \label{QBVM}
 The QBVM in \cite{Yang_2013} for regularizing (\ref{state})  solves the following well-posed
  regularized problem
   \eq \label{state_qbvm}
  \left\{\begin{array}{ll}
  	u_t -\Delta u =f,\ \qquad &\tn{in} \Omega\times (0,T),  \\
  	u(\cdot,t)=0, &\tn{on} \partial\Omega\times (0,T), \\
  	u(\cdot,0)=\phi , &\tn{in} \Omega, \\
  	u(\cdot,T){+\beta f(\cdot)}=g_\delta , &\tn{in} \Omega,
  \end{array}\right.
  \ee
  where $\beta>0$ is a regularization parameter to be chosen based on the noise level $\delta>0$.
  Compared with the Tikhonov regularization \cite{yang2010simplified} of minimizing a regularized functional $ \|Kf-g_\delta\|_2^2+ \gamma\|f\|_2^2$ with $K$ being a compact solution operator and $\gamma$ being a regularization parameter, the QBVM provides a better control of the system structure after discretization. In particular, the QBVM does not need  to explicitly construct  $K$ or its adjoint $K^*$ and use any eigenfunctions of the spatial differential operator.

Let $I_h\in \IR^{m\times m}$ be an identity matrix.
 With a center finite difference scheme in space (denotes $\Delta_h\in\IR^{m\times m} $ by the discrete Laplacian matrix with a uniform step size $h>0$) and a backward Euler scheme in time (with a uniform time step size $\tau=T/n$),  the full discretization of (\ref{state_qbvm}) reads (with the initial condition $u^0=\phi_h$ and $u^j\approx u(\cdot,j \tau)$ over all spatial grids)
 \begin{align} \label{state_qbvm_new_h}
 	\left\{\begin{array}{ll}
 		{(u^j-u^{j-1})}/{\tau} -\Delta_h u^j -f_h=0,\ j=1,2,\cdots, n, \\
 		u^n+{\beta f_h}=g_{\delta,h},
 	\end{array}\right.
 \end{align}
 which can be reformulated into a nonsymmetric sparse linear system
 \begin{align} \label{linsysQBVP}
 	\widehat A_h {\bm u}_h=  \widehat {\bm b}_h,
 \end{align}
 where %$\bm u_h=[f_h,u^1,u_2,\cdots,u^{n-1},u^{n}]^\T$ and
 \begin{align*}
 	\widehat A_h&=\bmt
 	{\beta I_h }& 0&0 &\cdots &0 &I_h\\
 	-{I_h}  & {I_h}/{\tau} -\Delta_h &0 & \cdots &0 &0\\
 	-{I_h}&-{I_h}/{\tau} &  {I_h}/{\tau}-\Delta_h & 0 & \cdots &0\\
 	\vdots&0&\ddots &\ddots &\ddots &0\\
 	-{I_h}& 0&\cdots &-{I_h}/{\tau} & {I_h}/{\tau}-\Delta_h & 0\\
 	-{I_h}&0&\cdots & 0 & {-I_h}/{\tau} & {I_h}/{\tau}-\Delta_h
 	\emt,
 	{\bm u}_h&=\bmt f_h\\ u^1 \\u^2 \\ \vdots \\ u^{n-1}\\ u^n \emt,
 	\widehat {\bm b}_h=\bmt g_{\delta,h}\\\phi_h/\tau\\ 0  \\ \vdots \\ 0\\0 \emt.
 \end{align*}
Clearly, the $(1,1)$ block $\beta I_h$ is different from the other diagonal blocks, which prevents a Kronecker product formulation of  $\widehat A_h$ desired in PinT algorithm as shown in the our new parameterized QBVM.

% To obtain a better convergence rate, a modified QBVM was proposed in \cite{wei2014modified},
% within the framework of time-fractional  diffusion equation, where the regularized problem (corresponding to our model) reads
%   \eq \label{state_qbvm2}
%  \left\{\begin{array}{ll}
%  	u_t -\Delta u =f,\ \qquad &\tn{in} \Omega\times (0,T),  \\
%  	u(\cdot,t)=0, &\tn{on} \partial\Omega\times (0,T), \\
%  	u(\cdot,0)=\phi , &\tn{in} \Omega, \\
%  	u(\cdot,T){-\beta\Delta f(\cdot)}=g_\delta , &\tn{in} \Omega.
%  \end{array}\right.
%  \ee
 \subsection{A new quasi-boundary value method based on finite difference scheme}
 To get a better structured linear system that allows a fast direct PinT solver upon finite difference discretization, we propose the following new parameterized QBVM (PQBVM)
  \eq \label{state_qbvm_new}
 \left\{\begin{array}{ll}
 	u_t -\Delta u =f,\ \qquad &\tn{in} \Omega\times (0,T),  \\
 	u(\cdot,t)=0, &\tn{on} \partial\Omega\times (0,T), \\
 	u(\cdot,0)=\phi , &\tn{in} \Omega, \\
 	u(\cdot,T)+{\beta(\alpha f(\cdot)-\Delta f(\cdot)})=g_\delta , &\tn{in} \Omega,
 \end{array}\right.
 \ee
 where $\alpha\ge 0$ is a free design parameter to control the condition number of the subsequent direct PinT solver. In general with $\alpha\ne 0$, we expect the above new PQBVM to have a similar convergence rate as the standard QBVM in \cite{Yang_2013} due to the shared term $f$. We highlight that the special choice of $\alpha=0$ in fact leads to the known modified QBVM (MQBVM) established in \cite{Wei_2014b} within the framework of time-fractional  diffusion equation.
 However, the authors in \cite{Wei_2014b} focused on studying the improved convergence rates of MQBVM, without discussing fast algorithms for solving the regularized linear systems. In this paper we propose the above PQBVM mainly from the perspective of designing regularized linear systems with better structures that are suitable for constructing direct PinT algorithms, while at the same time retaining the convergence rates of QBVM.
 We emphasize that $\alpha\ge 0$ should not be treated as another regularization parameter like $\beta>0$ and it will be chosen purely for facilitating the development of fast direct PinT system solvers.

 With the same center finite difference scheme in space and  backward Euler scheme in time as used in the above discretization (\ref{state_qbvm_new_h}),  the full discretization of (\ref{state_qbvm_new}) leads to
\begin{align} \label{state_pqbvm_new_h}
 \left\{\begin{array}{ll}
  {(u^j-u^{j-1})}/{\tau} -\Delta_h u^j -f_h=0,\ j=1,2,\cdots, n, \\
 	u^n+{\beta(\alpha f_h-\Delta_h f_h})=g_{\delta,h},
 \end{array}\right.
\end{align}
which, after dividing the last equation by $\beta$, can be reformulated into a nonsymmetric  linear system
\begin{align} \label{linsysQBVP}
	A_h {\bm u}_h=  {\bm b}_h,
\end{align}
where
\begin{align*}
	A_h&=\bmt
	{\alpha I_h-\Delta_h}& 0&0 &\cdots &0 &I_h/\beta\\
	-{I_h}  & {I_h}/{\tau} -\Delta_h &0 & \cdots &0 &0\\
	-{I_h}&-{I_h}/{\tau} &  {I_h}/{\tau}-\Delta_h & 0 & \cdots &0\\
	\vdots&0&\ddots &\ddots &\ddots &0\\
	-{I_h}& 0&\cdots &-{I_h}/{\tau} & {I_h}/{\tau}-\Delta_h & 0\\
	-{I_h}&0&\cdots & 0 & {-I_h}/{\tau} & {I_h}/{\tau}-\Delta_h
	\emt,
	%u_h=\bmt f_h\\ u^1 \\u^2 \\ \vdots \\ u^{n-1}\\ u^n \emt,
	{\bm b}_h=\bmt g_{\delta,h}/\beta\\\phi_h/\tau\\ 0  \\ \vdots \\ 0\\0 \emt.
\end{align*}
We can now rewrite the block-structured matrix $A_h$ in  (\ref{linsysQBVP}) into Kronecker product form
\begin{align} \label{Akron1}
 A_h= B \otimes I_h - I_t\otimes \Delta_h
\end{align}
where $I_t\in\IR^{(n+1)\times (n+1)}$ denotes an identity matrix and the time discretization matrix $B$ is given by
\begin{align}
	B &=\bmt
	\alpha& 0&0 &\cdots &0 &1/\beta\\
	-1 &  1/\tau &0 & \cdots &0 &0\\
	-1&-1/\tau &  1/\tau & 0 & \cdots &0\\
	\vdots&0&\ddots &\ddots &\ddots &0\\
	-1& 0&\cdots &-1/\tau& 1/\tau& 0\\
	-1&0&\cdots & 0 & -1/\tau &1/\tau
	\emt \in\IR^{(n+1)\times (n+1)}.
\end{align}	
Such a Kronecker product reformulation (\ref{Akron1}) is crucial to develop our following fast direct PinT solver,
%\red{where the diagonalization of the matrix $B$ is needed. Since $B$ is nonsymmetric and has a nontrivial structure, it is diagonalizable only under certain conditions, which will be discussed in Section \ref{diagofB}.}
 {which requires the matrix $B$ to be diagonalizable.  Since $B$ is nonsymmetric and has a nontrivial structure, its diagonalizability is not straightforward and will be discussed separately in Section \ref{diagofB}.}
\subsection{A diagonalization-based direct PinT solver}
\label{PinT}
Suppose $B$ has a diagonalization $B=V D V^{-1}$, where
$D=\text{diag}(d_{1},\dots,d_{n+1})$ with $d_j$ being the $j$-th eigenvalue of $B$
and the $j$-th column of the invertible matrix $V$ gives the corresponding eigenvector.  Then we can factorize $A_h$ into
the product form
\[
A_h=(VDV^{-1} )\otimes I_h - I_t\otimes \Delta_h=\underbrace{(V\otimes I_h)}_{\tn{Step-(a)}}\underbrace{\left( D\otimes I_h-I_t\otimes \Delta_h \right)}_{\tn{Step-(b)}} \underbrace{(V^{-1}\otimes I_h)}_{\tn{Step-(c)}}.
\]
Hence,  let $Z=\texttt{mat}({\bm b}_h)\in\IR^{m\times (n+1)}$,  the solution ${\bm u}_h= A_h^{-1} {\bm b}_h$ can be computed via three steps:
\begin{equation}\label{3step}
	\begin{split}
		&\text{Step-(a)} ~~S_1=  Z V^{-\T}  ,\\
		&\text{Step-(b)} ~~S_{2}(:,j)=\left( { d_{j}}I_h - \Delta_h\right)^{-1} S_{1}(:,j),\quad ~j=1,2,\dots,n+1,\\
		&\text{Step-(c)} ~~{\bm u}_h=\texttt{vec}(S_2V^\T)  ,\\
	\end{split}
\end{equation}
where  $S_{1,2}(:,j)$ denotes the $j$-th column of $S_{1,2}$ and $V^\T$ defines the non-conjugate transpose of $V$.
Here we have used the efficient Kronecker product property $(C \otimes I_h)\texttt{vec}(X)=\texttt{vec}( XC^\T)$ for any compatible matrices $C$ and $X$.
Clearly, the $(n+1)$ fully independent complex-shifted linear systems in Step-(b) can be computed in parallel.
Notice that a different spatial discretization only affects the  matrix $\Delta_h$ in Step-(b).

Let $\kappa_p(V)=\|V\|_p\|V^{-1}\|_p$ with $p=1,2,\infty$ denotes the matrix $p$-norm condition number of $V$.
Numerically, the overall round-off errors of such a 3-steps diagonalization-based PinT direct solver is proportional to the condition number
of $V$, see Lemma 3.2 in \cite{caklovic2021parallel} for a detail round-off error analysis.
Hence, it is essential to design the matrix $B$ so that the condition number of  $V$ is well controlled for stable computations. In particular, it would become numerically unstable if $\kappa(V)$ grows  exponentially with respect to $n$.
In view of the discretization errors in space and time, it is acceptable to have $\kappa(V)=\mathcal{O}(n^{q})$ with a small $q$ (say $q\le 3$).
%%%%%%%%%%%%%%%%%%%%%%%%%%%%%%%%%

 \section{The diagonalization of $B$ and the condition number of $V$}\label{diagofB}
% \subsection{The diagonalization of $B$}
 In this subsection we will prove that the matrix $B$ with a special choice of $\alpha$ is indeed diagonalizable and also provide explicit formulas for computing its eigenvector matrix $V$ and estimating its condition number. More specifically, we will prove that $\kappa_1(V)=\CO(cn)$ with $c=\beta/\tau^2$ under the special choice $\alpha=\alpha_*:=1/\tau+\tau/\beta$.
 Although the trivial choice of $\alpha=0$ can be numerically used in the diagonalization-based direct PinT solver, there is no theoretical guarantee that the corresponding matrix $B$ is diagonalizable
 and/or the eigenvector matrix $V$ is well-conditioned for stable computation.
 In particular, the corresponding analysis based on the trivial choice of $\alpha=0$ seems to be far too difficult to perform due to much more complicated eigenvalue/eigenvector expressions, which shows the necessity of introducing the free design parameter $\alpha$.

 Let $\lambda$ be an eigenvalue of $B$ with nonzero eigenvector $\bm v=[v_0,\cdots,v_n]^\T$. By $B\bm v=\lambda \bm v$ we have
 \begin{align}
   \alpha v_0+v_n/\beta&=\lambda v_0,\label{evp0}\\
   -v_0+v_1/\tau&=\lambda v_1,\label{evp1}
 \end{align}
 and
 \begin{equation}\label{evpk}
   -v_0-v_{k-1}/\tau+v_k/\tau=\lambda v_k,~~k=2,\cdots,n.
 \end{equation}
 Obviously, $v_0\neq0$ and $\lambda\neq 1/\tau$ since otherwise it leads to $\bm v=\bm 0$. Without loss of generality, we choose $v_0=1/\tau$. It is readily seen from \eqref{evp1} and \eqref{evpk} that
 \begin{equation} \label{eigveck}
   v_k=\mu+\cdots+\mu^k=\frac{\mu^{k+1}-\mu}{\mu-1},~~k=1,\cdots,n,
 \end{equation}
 where $\mu=1/(1-\tau\lambda)\ne 1$. Assume $\alpha=\alpha_*:=1/\tau+\tau/\beta$ and denote $c=\beta/\tau^2$.
 Substituting $v_0=1/\tau$ and the above formula for $v_n$ into the equation \eqref{evp0} yields
 \[
 \frac{1}{\tau^2}+\frac{1}{\beta}+\frac{1}{\beta}(\mu+\cdots+\mu^n)=\frac{\lambda}{\tau},
 \]
 which, up on multiplying both sides by $\beta\mu$, reduces to
 \begin{equation}\label{mu-eq}
   c+\mu+\cdots+\mu^{n+1}=0.
 \end{equation}
The $(n+1)$ roots of (\ref{mu-eq}) determine the $(n+1)$ eigenvalues of $B$.
 For convenience, we define
 \begin{equation}\label{phi}
   \psi(\mu):=(\mu-1)(c+\mu+\cdots+\mu^{n+1})=\mu^{n+2}+(c-1)\mu-c.
 \end{equation}
 We have the following result.
 \begin{lemma}\label{lem-evp}
   If $\alpha=\alpha_*:=1/\tau+\tau/\beta$ and $c=\beta/\tau^2>1$, then the matrix $B$ has $n+1$ distinct eigenvalues. In particular, this implies the nonsymmetric matrix $B$ is indeed diagonalizable.
 \end{lemma}
 \begin{proof}
   It suffices to prove that the equation \eqref{mu-eq} has no repeated roots.
   Assume to the contrary that $\mu=\mu_0$ is a repeated root of \eqref{mu-eq}.
   We then have $\psi(\mu_0)=\psi'(\mu_0)=0$. From $\psi'(\mu_0)=0$ we obtain
   $$\mu_0^{n+1}=(1-c)/(n+2).$$
   Substituting this into $\psi(\mu_0)=0$ gives
   $$\mu_0={c(n+2)\over(c-1)(n+1)}.$$
   Coupling the above two equations yields
   $$c^{n+1}(n+2)^{n+2}+(c-1)^{n+2}(n+1)^{n+1}=0,$$
   which contradicts to the condition $c>1$. This completes the proof.
 \end{proof}
 Denote by $\mu_1,\cdots,\mu_{n+1}$ the distinct roots of the equation \eqref{mu-eq}.
 The eigenvalues of $B$ are $\lambda_k=(1-1/\mu_k)/\tau$ with $k=1,\cdots,n+1$.
 The above eigenvector expression (\ref{eigveck}) implies that the eigenvector (after rescaled by $(\mu_k-1)$) corresponding to the eigenvalue $\lambda_k$ can be chosen as
 $$\bm v^{(k)}=[(\mu_k-1)/\tau,\mu_k^2-\mu_k,\cdots,\mu_k^{n+1}-\mu_k]^\T.$$ Hence, we have the eigendecomposition $B=V\diag\{\lambda_1,\cdots,\lambda_{n+1}\}V^{-1}$ with the eigenvector matrix
 \begin{equation}\label{V}
   V=\begin{bmatrix}
     (\mu_1-1)/\tau&\cdots&(\mu_{n+1}-1)/\tau\\
     \mu_1^2-\mu_1&\cdots&\mu_{n+1}^2-\mu_{n+1}\\
     \vdots&&\vdots\\
     \mu_1^{n+1}-\mu_1&\cdots&\mu_{n+1}^{n+1}-\mu_{n+1}
   \end{bmatrix}.
 \end{equation}
 We remark that  $V\Phi$  is also an eigenvector matrix  for any nonsingular diagonal matrix $\Phi$.
 % \subsection{The condition number of $V$}

The following lemma shows that the roots of the equation \eqref{mu-eq} are located in the annulus $1<|\mu|<(2c-1)^{1/(n+1)}$ on the complex plane.
This implies $|\tau\lambda_k-1|=\frac{1}{|\mu_k|}\in ((2c-1)^{-1/(n+1)},1)$.
\begin{lemma}\label{lem-mu-bound}
  Let $\mu_1,\cdots,\mu_{n+1}$ be distinct roots of the equation \eqref{mu-eq}. If $c>1$, then $|\mu_k|>1$ and $|\mu_k|^{n+1}<2c-1$ for $k=1,\cdots,n+1$.
\end{lemma}
\begin{proof}
  It is obvious from $c>1$ that $\mu_k\neq1$.
  We claim that $|\mu_k|>1$; otherwise, we obtain from $\psi(\mu_k)=0$ that
  $$c=\mu_k^{n+2}+(c-1)\mu_k\le|\mu_k|^{n+2}+(c-1)|\mu_k|\le c,$$
  which is satisfied if and only if $\mu_k=1$, a contradiction.
  Next, it follows from $\psi(\mu_k)=0$ and $|\mu_k|>1$ that
  $$|\mu_k^{n+1}|=|c/\mu_k-(c-1)|\le c/|\mu_k|+|c-1|<2c-1.$$
  The proof is completed.
\end{proof}
To find an explicit expression for the inverse matrix $W=V^{-1}$, we shall make use of the Lagrange interpolation polynomials (such that $L_j(\mu_l)=\delta_{j,l}$ with $\delta_{j,l}$ being the Kronecker delta)
\begin{equation}\label{Lj}
  L_j(\mu)=\prod_{1\le k\le n+1,k\neq j}{\mu-\mu_k\over\mu_j-\mu_k}=\sum_{k=1}^{n+1}L_{jk}\mu^{k-1},~~j=1,\cdots,n+1,
\end{equation}
where $L_{jk}=L_j^{(k-1)}(0)/(k-1)!$ is the coefficient of $\mu^{k-1}$ in the polynomial expression of $L_j(\mu)$.
Let $U=\left[U_{kl}\right]_{k,l=1}^{n+1}$ be the Vandermonde matrix with $U_{kl}=\mu_l^{k-1}$ and $L=\left[L_{jk}\right]_{j,k=1}^{n+1}$.
It follows from the identities $L_j(\mu_l)=\sum_{k=1}^{n+1}L_{jk}\mu_l^{k-1}=\delta_{j,l}$  that $LU=I$ with $I$ being an identity matrix of size $(n+1)$.
\begin{lemma}\label{lem-Wjk}
  Let $W=V^{-1}=\left[W_{jk}\right]_{j,k=1}^{n+1}$. We have the following expressions
  \begin{align}
    W_{jk}=\begin{cases}
      L_{j,k+1}-L_j(1)/(n+c+1),&~~j<n+1,~k>1,\\
     -L_j(1)/(n+c+1),&~~j=n+1,~k>1,\\
      c\tau L_j(1)/(n+c+1)-\tau L_{j1},&~~k=1.
    \end{cases}
  \end{align}
\end{lemma}
\begin{proof}
  First, we consider the case $k>1$. It follows from \eqref{V} and $VW=I$ that
  $$\sum_{j=1}^{n+1}(\mu_j-1)W_{jk}=0,$$
  and
  $$\sum_{j=1}^{n+1}(\mu_j^l-\mu_j)W_{jk}=\begin{cases}
    1,&~~l=k,\\
    0,&~~l>1,~l\neq k.
  \end{cases}$$
  For convenience, we denote $S_k=\sum_{j=1}^{n+1}W_{jk}$. It is readily seen from the above equations that
  \begin{equation}\label{Wjk-eq}
  \sum_{j=1}^{n+1}\mu_j^lW_{jk}=\begin{cases}
    S_k+1,&~~l=k,\\
    S_k,&~~l\neq k.
  \end{cases}.
  \end{equation}
  Recall that $\mu_1,\cdots,\mu_{n+1}$ are the roots of the equations \eqref{mu-eq}; namely, $c+\sum_{l=1}^{n+1}\mu_j^l=0$.
  We then obtain
  $$0=c\sum_{j=1}^{n+1}W_{jk}+\sum_{l=1}^{n+1}\sum_{j=1}^{n+1}\mu_j^lW_{jk}=
  cS_k+\left(\sum_{l=1,l\neq k}^{n+1}S_k\right)+(S_k+1)=cS_k+(n+1)S_k+1,$$
  which implies $S_k=-1/(n+c+1)$ is independent of $k$.
  Now, we multiply both sides of \eqref{Wjk-eq} by $L_{m,l+1}$ and then add from $l=0$ to $l=n$ to find
  \begin{align*}
    W_{mk}=&L_m(\mu_j)W_{jk}=\sum_{l=0}^nL_{m,l+1}S_k+\begin{cases}
      L_{m,k+1},&~~k<n+1,\\
      0&,~~k=n+1
    \end{cases}
    \\=&-{L_m(1)\over n+c+1}+\begin{cases}
      L_{m,k+1},&~~k<n+1,\\
      0&,~~k=n+1.
    \end{cases}
  \end{align*}
  Next, we consider the case $k=1$. It follows from \eqref{V} and $VW=I$ that
  $$\sum_{j=1}^{n+1}(\mu_j-1)W_{j1}=\tau,$$
  and
  $$\sum_{j=1}^{n+1}(\mu_j^l-\mu_j)W_{j1}=0,~~l>1.$$
  For convenience, we denote $S_1=\sum_{j=1}^{n+1}W_{j1}$. It is readily seen from the above equations that
  \begin{equation}\label{Wj1-eq}
  \sum_{j=1}^{n+1}\mu_j^lW_{jk}=\begin{cases}
    S_1,&~~l=0,\\
    S_1+\tau,&~~l=1,\cdots,n+1.
  \end{cases}.
  \end{equation}
  Recall that $\mu_1,\cdots,\mu_{n+1}$ are the roots of the equations \eqref{mu-eq}.
  We then obtain
  $$0=c\sum_{j=1}^{n+1}W_{jk}+\sum_{l=1}^{n+1}\sum_{j=1}^{n+1}\mu_j^lW_{jk}=cS_1+(n+1)(S_1+\tau),$$
  which implies $S_1=-(n+1)\tau/(n+c+1)$.
  Now, we multiply both sides of \eqref{Wj1-eq} by $L_{m,l+1}$ and then add from $l=0$ to $l=n$ to find
  \begin{align*}
    W_{m1}=&L_m(\mu_j)W_{jk}=L_{m1}S_1+\sum_{l=1}^nL_{m,l+1}(S_1+\tau)
    =-\tau L_{m1}+{c\tau L_m(1)\over n+c+1}.
  \end{align*}
  This completes the proof.
\end{proof}
The following lemma gives an explicit formula for $L_{jk}$, which will be used to estimate $\|W\|_1$.
\begin{lemma}\label{lem-Ljk}
  Let $L_{jk}$ with $1\le j,k\le n+1$ be the coefficient of $\mu^{k-1}$ in the polynomial expression of Lagrange interpolation polynomial $L_j(u)$ defined in \eqref{Lj}. We have
  \begin{equation}\label{Ljk}
    L_{jk}={\mu_j^{n+2-k}-1\over(n+2)\mu_j^{n+1}+c-1}={\mu_j^{1-k}-\mu_j^{-n-1}\over n+2+(c-1)\mu_j^{-n-1}}.
  \end{equation}
\end{lemma}
\begin{proof}
  Recall from \eqref{phi} that $\psi(\mu)=(\mu-1)(c+\mu+\cdots+\mu^{n+1})$.
  Since $\mu_1,\cdots,\mu_{n+1}$ are distinct roots of the polynomial equation $c+\mu+\cdots+\mu^{n+1}=0$, we can factor the polynomial as
  $$c+\mu+\cdots+\mu^{n+1}=\prod_{k=1}^{n+1}(\mu-\mu_k).$$
  Consequently,
  $\psi(\mu)=(\mu-1)(\mu-\mu_1)\cdots(\mu-\mu_{n+1})$ and
  $${\psi(\mu)\over(\mu-1)(\mu-\mu_j)}=\prod_{1\le k\le n+1,k\neq j}(\mu-\mu_k).$$
  We denote
  $$a_j=\prod_{1\le k\le n+1,k\neq j}(\mu_j-\mu_k)=\lim_{\mu\to\mu_j}{\psi(\mu)\over(\mu-1)(\mu-\mu_j)}={\psi'(\mu_j)\over\mu_j-1}.$$
  On the other hand, we obtain from \eqref{Lj} that
  \begin{align*}
    &c+\mu+\cdots+\mu^{n+1}=\prod_{k=1}^{n+1}(\mu-\mu_k)=a_j(\mu-\mu_j)L_j(u)=a_j(\mu-\mu_j)\sum_{k=1}^{n+1}L_{jk}\mu^{k-1}
    \\=&a_j[L_{j,n+1}\mu^{n+1}+(L_{jn}-\mu_jL_{j,n+1})\mu^n+\cdots+(L_{j1}-\mu_jL_{j2})\mu-\mu_jL_{j1}].
  \end{align*}
  Comparing the polynomial coefficients on both side of the equation gives
  $c=-a_j\mu_jL_{j1}$, and
  $$1=a_j(L_{j1}-\mu_jL_{j2})=\cdots=a_j(L_{jn}-\mu_jL_{j,n+1})=a_jL_{j,n+1}.$$
  It is readily seen that
  $$a_jL_{jk}=1+\mu_j+\cdots+\mu_j^{n+1-k}={\mu_j^{n+2-k}-1\over\mu_j-1}.$$
  This together with $a_j=\psi'(\mu_j)/(\mu_j-1)$ and \eqref{phi} proves \eqref{Ljk}.
\end{proof}
To show that $|L_{jk}|=O(1/n)$ uniformly for all $1\le j,k\le n+1$, we need the following lemma.
\begin{lemma}\label{lem-Ljk-bound}
  Assume $c>1$ and $n>11$. Let $\mu_1,\cdots,\mu_{n+1}$ be the distinct roots of \eqref{mu-eq}. We have
  $|n+2+(c-1)\mu_j^{-n-1}|>n/2$ for all $j=1,\cdots,n+1$.
\end{lemma}
\begin{proof}
  We will prove by contradiction. Assume to the contrary that
  $|n+2+(c-1)\mu_j^{-n-1}|\le n/2$ for some $\mu_j=re^{i\theta}$ with $r>1$ and $\theta\in[0,\pi]$.
  Let $b=r^{n+1}/(c-1)>0$, we then have
  \[
  \left|b(n+2)+\cos[(n+1)\theta]-i\sin[(n+1)\theta]\right|\le bn/2,
  \]
  which gives
  \begin{align} \label{ieq2}
    \cos[(n+1)\theta]<2b+\cos[(n+1)\theta]&<-{bn\over2},\qquad
    |\sin[(n+1)\theta]| \le{bn\over2}.
  \end{align} 
  Note from \eqref{mu-eq} and \eqref{phi} that
  $\mu_j^{n+1}=c/\mu_j-(c-1)$, which upon dividing both sides by $(c-1)$ gives
  \[
  b \left(\cos[(n+1)\theta]+i\sin[(n+1)\theta]\right)=
  a(\cos\theta-i\sin\theta)-1,
  \]
  where $a=c/[r(c-1)]>0$. Hence we have
  \begin{align} \label{eq3}
    b&=\sqrt{1-2a\cos\theta+a^2},\quad
    b\cos[(n+1)\theta] =a\cos\theta-1,\quad
    b\sin[(n+1)\theta] =-a\sin\theta.
  \end{align} 
  The third equality together with $\theta\in[0,\pi]$ implies $\sin[(n+1)\theta]=-(a/b)\sin\theta\le 0$, which  gives $\theta\ge\pi/(n+1)$.
  We further obtain from the three equalities in (\ref{eq3}) and two inequalities in (\ref{ieq2}) that
  \begin{align*}
    &0<1-2a\cos\theta+a^2=b^2<{b^2n\over 2}<-b\cos[(n+1)\theta]=1-a\cos\theta<b<-{2\over n}\cos[(n+1)\theta]<{2\over n}
  \end{align*}
and hence (note the inequality $1-2a\cos\theta+a^2<1-a\cos\theta$ gives $a<\cos\theta$)
 \begin{align}
a\sin\theta=b|\sin[(n+1)\theta]|\le{b^2n\over2}<{2\over n}, \qquad
1-{2\over n}<a\cos\theta<a< \cos\theta.
\end{align}
  Consequently, we obtain $(n-2)<na$, $\theta\in[\pi/(n+1),\pi/2)$ such that $\theta<\tan\theta$,  and
   \begin{align}
  {\pi(n-2)\over n(n+1)}<\theta\cos\theta<\sin\theta<{2\over na}<{2\over n-2},
\end{align}
 which is not true for $n>11$ and hence contradicts our assumption.
  This completes our proof.
\end{proof}
Finally, we are ready to estimate the condition number of the eigenvector matrix $V$ in \eqref{V}.
\begin{theorem}
  If $\alpha=\alpha_*:=1/\tau+\tau/\beta$ and $c=\beta/\tau^2>1$, then $\kappa_1(V)=\|V\|_1 \|W\|_1=\CO(cn)$.
\end{theorem}
\begin{proof}
  Lemma \ref{lem-mu-bound} implies $|\mu_k|^j<2c-1$ for any $1\le j,k\le n+1$.
  It is easily seen from \eqref{V} that
  $$\sum_{j=1}^{n+1}|V_{jk}|=|\mu_k-1|/\tau+\sum_{j=2}^{n+1}|\mu_k^j-\mu_k|\le (2c)/\tau+(4c-2)n= (2c/T)n+(4c-2)n.$$
  In particular, $\|V\|_1=\CO(cn)$.
  Lemma \ref{lem-mu-bound} also implies $|\mu_j|>1$ for any $1\le j\le n+1$.
  Assume $n>11$. It then follows from Lemma \ref{lem-Ljk} and Lemma \ref{lem-Ljk-bound} that
  $$|L_{jk}|\le{|\mu_j|^{1-k}+|\mu_j|^{-n-1}\over |n+2+(c-1)\mu_j^{-n-1}|}<{1+1\over n/2}=4/n$$
  for all $1\le j,k\le n+1$. This together with Lemma \ref{lem-Wjk} and $|L_j(1)|=|\sum_{k=1}^{n+1}L_{jk}|<\frac{4(n+1)}{n}$ yields
    \begin{align*}
    W_{jk}\le\begin{cases}
      |L_{j,k+1}|+{|L_j(1)|\over n+c+1}\le {4\over n}+{4(n+1)\over n(n+c+1)}\le{8\over n},&~~j<n+1,~k>1,\\
     {|L_j(1)|\over n+c+1}\le{4(n+1)\over n(n+c+1)}<{4\over n},&~~j=n+1,~k>1,\\
      {c\tau |L_j(1)|\over n+c+1}+\tau |L_{j1}|\le {4c\tau(n+1)\over n(n+c+1)}+{4\tau\over n}<{4 \tau(n+1)\over n }+{4\tau\over n}={4\tau(n+2)\over n},&~~k=1.
    \end{cases}
  \end{align*}
  The above inequalities can be combined into
  $|W_{jk}|<[8+4\tau(n+2)]/n$ for all $1\le j,k\le n+1$ with $n>11$. In particular, $\|W\|_1=\CO(1)$.
  Therefore, the condition number of the eigenvector matrix $V$ with respect to the matrix $1$-norm is
  $\kappa_1(V)=\|V\|_1 \|W\|_1=\CO(cn)$.
  This completes the proof.
\end{proof}

%%%%%%%%%%%%%%%%%%%%%%%%%%%%%%%%
Our subsequent convergence analysis shows that with $\alpha=\alpha_*$ the choice of regularization parameter $\beta=\tau\delta^{1/2}$  gives an $\CO(\delta^{1/2})$ convergence rate, which yields a provable condition number estimate  $\kappa_1(V)=\CO(\delta^{1/2}n/\tau)=\CO(\delta^{1/2}n^2)$.
We remark that the trivial choice $\alpha=0$ may also work well in numerical,
but the corresponding condition number of $V$ can be larger and  it is also more difficult to estimate due to very complicated characteristic equations for the eigenvalues of $B$.

  Figure \ref{FigCondV} illustrates the two very different growth rates of the condition number of $V$ corresponding to the MQBVM (with $\alpha=0,\beta=\delta$) and our PQBVM (with $\alpha=\alpha_*,\beta=\tau\delta^{1/2}$), respectively.
For a large mesh size $n$ and noise level $\delta$, the condition number of $V$ with $\alpha=\alpha_*$ is indeed several order of magnitude smaller than that with $\alpha=0$, which also numerically validated the estimated condition number growth rate with the optimized choice $\alpha=\alpha_*$.
\begin{figure}[htp!]
	\begin{center}
		\includegraphics[width=1\textwidth]{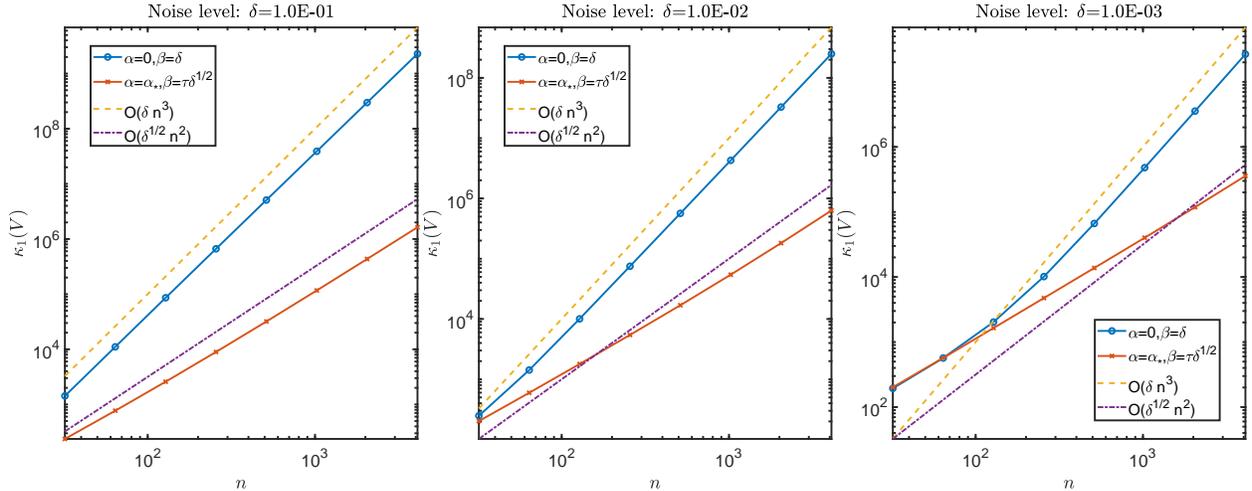}
	\end{center}
	\vspace{-2em}
	\caption{Comparison of the condition number $\kappa_1(V)$ and its estimated bounds with two different choices of $\alpha$ and $\beta$.}
	
	\label{FigCondV}
\end{figure}

\section{Convergence analysis}
In this section, we will analyze the convergence rate of our proposed PQBVM, where the optimized parameter $\alpha=\alpha_*=1/\tau+\tau/\beta$ leads to a mesh-dependent regularization parameter $\beta>0$. We emphasize that the presented analysis is different from the MQBVM \cite{Wei_2014b} case with $\alpha=0$.

Let $\mathbb{A}=-\Delta$ and define a Hilbert function space $H=H^1_0(\Omega)$ equipped with the standard $L^2$ norm $\|f\|_2:=(f,f)^{1/2}=\left(\int _{\Omega}f^2d\bm x\right)^{\frac{1}{2}}$. Then the self-adjoint operator $\mathbb{A}$ admits a set of orthonormal eigenfunctions $\{X_l\}_{l\geq 1}$ in $H$, associated to a set of eigenvalues $\{\lambda _l\}_{l\geq 1}$ such that $\mathbb{A}X_l=\lambda_l X_l$ with $0<\lambda _1<\lambda _2<\cdots$ and $\lim \limits_{l\rightarrow \infty} \lambda_l=+\infty$. Given any $g\in H$, it has a series expansion $g=\sum ^{\infty}_{l=1}(g,X _l)X _l$, with $(g, X_l):=\int _{\Omega}g X_l d\bm x$ for all $l$.
We assume the measured data $g_{\delta}(x)\in L^2(\Omega)$ and it satisfies
\begin{align}\label{gBnd}
\|g-g_{\delta}\|_2\leq \delta.
\end{align}
We also impose \textit{a priori} bound for the heat source, that is,
\begin{align}\label{fBnd}
\|f\| _{H^p(\Omega)}:=\left(\sum ^{\infty}_{l=1}\lambda ^p_l(f,X _l)^2 \right)^{\frac{1}{2}} \leq E_f, \quad p\geq 0,
\end{align}
where $E_f>0$ is a constant. In particular, when $p=0$, (\ref{fBnd}) is reduced to the $L^2$ norm, that is
\begin{align}\label{normexpansion}
\|f\| _{H^0(\Omega)}=\left(\sum ^{\infty}_{l=1}(f,X_l)^2\right)^{1/2}= \|\sum ^{\infty}_{l=1}(f,X _l)X _l\|_2=\|f\| _2 =\left(\int _{\Omega}f^2dx\right)^{\frac{1}{2}}.
\end{align}
Consider the exact  noisy-free problem (\ref{state}), by separation of variables and the initial condition $\phi=\sum ^{\infty}_{l=1}(\phi,X _l)X _l$, the unknown solution function $u$ can be expressed as (by solving the sequence of separated ODE initial value problem: $u_l'(t)+\lambda_l u_l(t)=(f,X_l)$ with $u_l(0)=(\phi,X_l)$)
\begin{align}
u(\cdot,t)=\sum ^{\infty}_{l=1} u_l(t)X_l=\sum ^{\infty}_{l=1}e^{-\lambda _lt}\left((f,X_l)\frac{e^{\lambda _lt}}{\lambda _l}+c_l\right)X_l,\quad c_l= {(\phi, X_l)}-\frac{{(f,X_l)}}{\lambda _l}.
\end{align}
Applying the final time condition $u(\cdot,T)=g=\sum ^{\infty}_{l=1}(g,X _l)X _l$, we further obtain
\begin{align}
u(\cdot,T)=\sum ^{\infty}_{l=1}\left(\frac{1-e^{-\lambda _lT}}{\lambda _l}(f,X_l)+e^{-\lambda _lT}(\phi, X_l) \right)X_l =\sum ^{\infty}_{l=1}(g,X_l)X_l=g,
\end{align}
which gives the exact source expression
\begin{align}\label{exsol}
f=\sum ^{\infty}_{l=1}(f,X_l) X_l,\quad \tn{with} (f,X_l)=\frac{\lambda _l}{1-e^{-\lambda _lT}}\left ((g,X_l)-e^{-\lambda _lT}(\phi,X_l) \right).
\end{align}
Clearly, this exact formula (\ref{exsol}) is unstable for reconstructing $f$ with a noisy $g_\delta$ since $\lambda _l\to\infty$ will magnify the noise, unless certain noise filters or regularization techniques are incorporated.
Similarly, we can obtain the representation for the regularized solutions. See (\ref{noisefrsol}) and (\ref{noisysol}) below.

Now we give the error estimate between the regularization solution and the exact solution.

\begin{theorem}
\label{thmcov}
Let $f^{\delta}_{\alpha,\beta}(x)$ be the regularization solution of the problem (\ref{state_qbvm_new}) with the measured data $g_{\delta}$ satisfying (\ref{gBnd}). Let $f(x)$ be the exact solution of the problem (\ref{state}) and satisfy \textit{a priori} condition (\ref{fBnd}) for any $p\geq0$. Then, by fixing $\alpha =\alpha_*=1/\tau+\tau/\beta$, there holds
\begin{enumerate}
\item[(1)] for $0<p<2$, if we choose $\beta=\tau \left(\frac{\delta}{E_f}\right)^{\frac{2}{p+2}}$, we have
\begin{align}
\label{case1bd}
\|f^{\delta}_{\alpha,\beta} - f\|_2 \leq C_1E^{\frac{2}{p+2}}_f \delta ^{\frac{p}{p+2}}+ O(\tau ^{\frac{p}{2}});
\end{align}
\item[(2)] for $2\le p<4$, if we choose $\beta=\tau\left(\frac{\delta}{E_f}\right)^{\frac{1}{2}}$, we have
\begin{align}
\label{case2bd}
\|f^{\delta}_{\alpha,\beta} - f\|_2 \leq \left(1+C_2(E_f\delta)^{\frac{1}{2}} \max \left\{1, \left(\frac{\delta}{E_f} \right)^{\frac{p-2}{4}} \right\}  \right) (E_f\delta)^{\frac{1}{2}}  +O(\tau);
\end{align}
\item[(3)] for $p\ge 4$, if we choose $\beta =\frac{\tau}{\sqrt{\tau +1}}\left(\frac{\delta}{E_f}\right)^{\frac{1}{2}}$, we have
\begin{align}
\label{case3bd}
\|f^{\delta}_{\alpha,\beta} - f\|_2 \leq C_3\sqrt{\tau +1}(E_f \delta)^{\frac{1}{2}}+O(\tau);
\end{align}
\end{enumerate}
where $C_1, C_2, C_3$ are positive constants that only depend on $p$, $T$, and $\lambda_1$.
\end{theorem}

\begin{proof}
Let $f_{\alpha, \beta}$ be the noise-free regularization solution. By the triangular inequality,  we have
\begin{align}
\|f^{\delta}_{\alpha,\beta}-f\|_2 \leq \|f^{\delta}_{\alpha,\beta}-f_{\alpha,\beta}\|_2 + \|f_{\alpha,\beta}-f\|_2,
\end{align}
where each term will be estimated separately based on the corresponding series expression.

By the separation of variables and the given side conditions, we can verify the following expressions
\begin{align}
%f(x)&=\sum ^{\infty}_{l=1}\frac{\lambda _l}{1-e^{-\lambda _lT}}\left((g,X_l)-e^{-\lambda_lT}(\phi,X_l)\right)X_l,\\
%f_{\alpha,\beta}(x)&=\sum ^{\infty}_{l=1}\frac{\lambda _l}{1-e^{-\lambda _lT}+\alpha \beta\lambda _l+\beta \lambda^2_l}\left((g,X_l)-e^{-\lambda_lT}(\phi,X_l)\right)X_l,\\
%f^{\delta}_{\alpha,\beta}(x)&=\sum ^{\infty}_{l=1}\frac{\lambda _l}{1-e^{-\lambda _lT}+\alpha \beta\lambda _l+\beta \lambda^2_l}\left((g_{\delta},X_l)-e^{-\lambda_lT}(\phi,X_l)\right)X_l.
\label{noisefrsol}
(f_{\alpha,\beta},X_l)=&\frac{\lambda _l}{1-e^{-\lambda _lT}+\alpha \beta\lambda _l+\beta \lambda^2_l}\left((g,X_l)-e^{-\lambda_lT}(\phi,X_l)\right)\\
\label{noisysol}
(f^{\delta}_{\alpha,\beta},X_l)=&\frac{\lambda _l}{1-e^{-\lambda _lT}+\alpha \beta\lambda _l+\beta \lambda^2_l}\left((g_{\delta},X_l)-e^{-\lambda_lT}(\phi,X_l)\right).
\end{align}

Then it holds that
\begin{align*}
\| f^{\delta}_{\alpha, \beta}-f_{\alpha, \beta}\|_2
=&\|\sum ^{\infty}_{l=1} \frac{\lambda _l}{1-e^{-\lambda _lT}+\alpha \beta \lambda _l+\beta \lambda ^2_l}(g_{\delta}-g,X_l)X_l \|_2\\
\leq &\sup _{l\geq 1}\left(  \frac{\lambda _l}{1-e^{-\lambda _lT}+\alpha \beta \lambda _l+\beta \lambda ^2_l} \right) \|g_{\delta}-g\|_2\\
\leq &\sup _{l\geq 1}\left(  \frac{1}{\frac{\gamma _1}{\lambda _l}+\alpha \beta+\beta \lambda _l} \right) \delta
\leq  \frac{\delta}{2\sqrt{\gamma _1\beta}+\alpha \beta},
\end{align*}
where $\gamma _1 := 1-e^{-\lambda _1T}>0$. When $\alpha = \alpha _* = 1/\tau +\tau /\beta$, we have
\begin{align}
\label{noiserbd}
\| f^{\delta}_{\alpha, \beta}-f_{\alpha, \beta}\|_2 \leq \frac{\delta}{2\sqrt{\gamma _1\beta}+\beta /\tau +\tau}\leq \frac{\tau \delta}{\beta}.
\end{align}
Meanwhile, based on (\ref{exsol}) and (\ref{noisefrsol}) , the error between the noise-free regularized solution and the exact solution satisfies
\begin{align*}
 \|f_{\alpha, \beta}-f\|_2
\leq & \left\Vert \sum ^{\infty}\left (\frac{\lambda _l}{1-e^{-\lambda _lT}+\alpha \beta\lambda _l+\beta \lambda^2_l} -\frac{\lambda _l}{1-e^{-\lambda _lT}} \right)\left((g,X_l)-(\phi, X_l)e^{-\lambda _lT} \right)X_l \right\Vert _2 \\
=&\left( \sum ^{\infty}_{l=1} \left (\frac{\lambda _l}{1-e^{-\lambda _lT}+\alpha \beta\lambda _l+\beta \lambda^2_l} -\frac{\lambda _l}{1-e^{-\lambda _lT}} \right)^2\left((g,X_l)-(\phi, X_l)e^{-\lambda _lT} \right)^2 \right)^{\frac{1}{2}}\\
=&\left( \sum ^{\infty}_{l=1} \left(\frac{\alpha \beta\lambda _l +\beta \lambda ^2_l}{1-e^{-\lambda _lT}+\alpha \beta\lambda _l+\beta \lambda^2_l} \right)^2 \frac{1}{\lambda ^p_l}\frac{\lambda ^p_l\lambda ^2_l}{\left(1-e^{-\lambda _lT} \right)^2}\left((g,X_l)-(\phi, X_l)e^{-\lambda _lT} \right)^2 \right)^{\frac{1}{2}}\\
\leq& \left(\sup _{l\geq 1}A_l\right) \left( \sum ^{\infty}_{l=1} \lambda ^p_l (f,X_l)^2\right)^{\frac{1}{2}}
= \left(\sup _{l\geq 1}A_l \right)\|f\| _{H^p(\Omega)}
\leq  \left(\sup _{l\geq 1}A_l\right) E_f,
\end{align*}
where
\begin{align*}
A_l=&\frac{\alpha \beta\lambda _l +\beta \lambda ^2_l}{(1-e^{-\lambda _lT}+\alpha \beta\lambda _l+\beta \lambda^2_l)\lambda ^{\frac{p}{2}}_l}
=  \frac{\alpha \beta\lambda ^{1-\frac{p}{2}}_l +\beta \lambda ^{2-\frac{p}{2}}_l}{1-e^{-\lambda _lT}+\alpha \beta\lambda _l+\beta \lambda^2_l}
\leq  \frac{\alpha \beta \lambda ^{1-\frac{p}{2}}_l}{\gamma _1 +\alpha \beta \lambda _l} + \frac{\beta \lambda ^{2-\frac{p}{2}}_l}{\gamma _1 +\beta \lambda ^2_l}.
\end{align*}
According to Lemma 2.7 in \cite{Wei_2014b}, we can obtain
\begin{align*}
\frac{\alpha \beta \lambda ^{1-\frac{p}{2}}_l}{\gamma _1 +\alpha \beta \lambda _l} \leq
\begin{cases}
C_4(\alpha \beta) ^{\frac{p}{2}}, \quad &0<p<2,\\
C_5\alpha \beta, \quad &p\geq 2,
\end{cases}
\quad \text{ and } \quad
 \frac{\beta \lambda ^{2-\frac{p}{2}}_l}{\gamma _1 +\beta \lambda ^2_l}\leq
\begin{cases}
C_6\beta ^{\frac{p}{4}}, \quad &0<p<4,\\
C_7\beta, \quad &p\geq 4,
\end{cases}
\end{align*}
where $C_i$, $i=4,5,6,7$, are positive constants that only depend on $p$, $T$, and $\lambda _1$, which leads to
\begin{align}
\label{Albd}
A_l \leq
\begin{cases}
C_4(\alpha \beta) ^{\frac{p}{2}} + C_6\beta ^{\frac{p}{4}},  \quad &0<p<2, \\
C_5\alpha \beta + C_6\beta ^{\frac{p}{4}}, \quad &2\leq p<4,\\
C_5\alpha \beta + C_7\beta, \quad &p\geq 4.
\end{cases}
\end{align}
For $\alpha = \alpha _* = 1/\tau +\tau /\beta$,  combining (\ref{noiserbd}), (\ref{Albd}) and the fact  $\alpha \beta=\beta/\tau+\tau \geq 2\sqrt{\beta}$, we show the desired error estimates in the following three different cases depending on the range of $p$:\\

\underline{\textit{Case (i)}}: when $0<p<2$, we have (due to $(a+b)^p\le 2^p(a^p+b^p)$ for any $a>0,b>0, p>0$)
\begin{align*}
A_l\leq C_4(\alpha \beta) ^{\frac{p}{2}} + C_62^{-\frac{p}{2}}(\alpha \beta)^{\frac{p}{2}}
\leq \tilde{C}_1(\beta /\tau+\tau) ^{\frac{p}{2}} \leq \tilde{C}_12^{\frac{p}{2}}((\beta / \tau)^{\frac{p}{2}} +\tau ^{\frac{p}{2}}),
\end{align*}
then it holds that
\begin{align}
\|f^{\delta}_{\alpha, \beta} -f\|_2
%\leq &\frac{\delta}{2\sqrt{\gamma _1\beta}+\alpha \beta} + \left(C_4(\alpha \beta) ^{\frac{p}{2}} + C_6\beta ^{\frac{p}{4}}\right)E_f\\
%\leq &\frac{\delta}{2\sqrt{\gamma _1\beta}+\beta /\tau+\tau} + \tilde{C}_1(\beta /\tau+\tau) ^{\frac{p}{2}} E_f \\
\leq \frac{\tau \delta}{\beta} + \tilde{C}_12^{\frac{p}{2}}((\beta / \tau)^{\frac{p}{2}} +\tau ^{\frac{p}{2}} )  E_f
\leq \frac{\tau \delta}{\beta} +\tilde{C}_2 (\beta / \tau)^{\frac{p}{2}} E_f +O(\tau ^{\frac{p}{2}}),
\end{align}
which, upon choosing $\beta=\tau \left(\frac{\delta}{E_f}\right)^{\frac{2}{p+2}}$ such that $\frac{\tau \delta}{\beta}=(\beta / \tau)^{\frac{p}{2}} E_f$, gives   the desired error estimate as in (\ref{case1bd}) with $C_1=1+\tilde{C}_2$. Here $\tilde{C}_1, \tilde{C}_2$ are positive constants that only depend on  $p$, $T$, and $\lambda _1$.\\

\underline{\textit{Case (ii)}}: when $2\leq p<4$, we have
\begin{align}
A_l\leq C_5\alpha \beta + C_62^{-\frac{p}{2}}(\alpha \beta)^{\frac{p}{2}},
\end{align}
and therefore
\begin{align*}
\|f^{\delta}_{\alpha, \beta} -f\|_2
\leq &\frac{\tau \delta}{\beta} + \left(C_5(\beta /\tau +\tau) + C_62^{-\frac{p}{2}}(\beta /\tau +\tau)^{\frac{p}{2}}\right)E_f\\
\leq &\frac{\tau \delta}{\beta} + \left(C_5(\beta /\tau +\tau) + C_6 ((\beta / \tau)^{\frac{p}{2}} +\tau ^{\frac{p}{2}}) \right)E_f\\
\leq &\frac{\tau \delta}{\beta} +C_2 \max \{\beta /\tau , (\beta /\tau)^{\frac{p}{2}} \}E_f +O(\tau+\tau ^{\frac{p}{2}}),
\end{align*}
where $C_2>0$ is a constant. By taking $\beta = \tau \left(\frac{\delta}{E_f} \right)^{\frac{1}{2}}$ such that $\frac{\tau \delta}{\beta} =(\beta /\tau) E_f$, we have
\begin{align*}
\|f^{\delta}_{\alpha, \beta} -f\|_2
\leq & (E_f\delta)^{\frac{1}{2}} + C_2 \max \left\{\left(\frac{\delta}{E_f} \right)^{\frac{1}{2}}, \left(\frac{\delta}{E_f} \right)^{\frac{p}{4}} \right\}E_f+O(\tau)\\
\leq & (E_f\delta)^{\frac{1}{2}} \left(1+C_2(E_f\delta)^{\frac{1}{2}} \max \left\{1, \left(\frac{\delta}{E_f} \right)^{\frac{p-2}{4}} \right\}  \right) +O(\tau),
\end{align*}
which proves the estimate (\ref{case2bd}).

\underline{\textit{Case (iii)}}: when $p\geq4$, we have
\begin{align}
\|f^{\delta}_{\alpha, \beta} -f\|_2
%\leq &\frac{\delta}{2\sqrt{\gamma _1\beta}+\alpha \beta} + (C_5\alpha \beta + C_7\beta)E_f\\
\leq &\frac{\tau}{\beta}\delta +(C_5(\beta/\tau +\tau)+ C_7\beta)E_f
\leq  \frac{\tau}{\beta}\delta +\tilde{C}_4(1/\tau +1)\beta E_f +O(\tau),
\end{align}
where $\tilde{C}_4>0$ is a constant. The error estimate in (\ref{case3bd}) is achieved with $C_3=1+\tilde{C}_4$ if we choose $\beta =\frac{\tau}{\sqrt{\tau +1}}\left(\frac{\delta}{E_f}\right)^{\frac{1}{2}}$ such that $\frac{\tau}{\beta}\delta=(1/\tau +1)\beta E_f$.
\end{proof}

\begin{remark}
For $p>0$, Theorem \ref{thmcov} indicates that $\| f^{\delta}_{\alpha, \beta} -f\|_2 \rightarrow 0$ as $\delta \rightarrow 0$, and the convergence rate depends on the regularity of $f$ (i.e. $p>0$). In particular,  for $p\ge 4$, the obtained convergence rate $O(\delta^{\frac{1}{2}})$ is slightly slower than the derived convergence rate $O(\delta^{\frac{2}{3}})$ of MQBVM \cite{Wei_2014b}.
This is reasonable since our PQBVM uses a nonzero $\alpha f(\cdot)$ term to control the condition number of $V$.
\end{remark}

\begin{remark}
%\red{The assumption $\alpha \beta <1$ is reasonable, considering that it's just a perturbation parameter to the original problem. In fact, in practice, $\alpha \beta$ is usually in the scale of ???, which can be observed in our numerical results in the next section. }
For $p=0$, we have $A_l\leq 1$, then only the boundedness of $\| f^{\delta}_{\alpha, \beta} -f\|_2$ can be ensured.
\end{remark}

 \section{Numerical examples} \label{secNum}
   In this section, we present some numerical examples to illustrate the computational efficiency of our proposed PQBVM method. All simulations are  implemented in serial with MATLAB on {a Dell Precision 5820 Workstation with Intel(R) Core(TM) i9-10900X CPU@3.70GHz CPU and 64GB RAM},
  where CPU times (in seconds) are estimated by the timing functions \texttt{tic/toc}.
  In QBVM, we directly solve the full sparse linear systems with
  MATLAB's backslash sparse direct solver, which runs very fast for several thousands (but not millions) of unknowns.
  Our proposed PQBVM (including MQBVM as a special case) will be solved by the 3-steps fast direct PinT solver (\ref{3step}), where the independent complex-shifted sparse linear systems in Step-(b) can be solved by fast direct solvers
  (Thomas' algorithm for 1D cases and FFT solver for 2D cases) for rectangular domains with regular grids. The diagonalization of $B=VDV^{-1}$ is computed with MATLAB's \texttt{eig} function
  and Step-(a) $ZV^{-\T}$ is done (with MATLAB code: \texttt{Z/(V.')}) by MATLAB's slash('/') direct solver.

  To avoid inverse crimes, for a given exact source $f$ we first solve the forward (direct) problem with Crank-Nicolson time-stepping scheme to compute $g$ and then generate the noisy final condition measurement by $g_{\delta}=g\times (1+\epsilon\times \rm{rand}(-1,1)),$
  where $\epsilon> 0$ controls the noise level
  and $\rm{rand}(-1,1)$ denotes random noise uniformly distributed within $[-1,1]$. We then further compute the estimated noise bound $\delta:=\|g^{\delta}-g\|_2$.
  In practice, the obtained noise bound $\delta$ may be over-estimated or under-estimated.
  Since $E_f$ in Theorem \ref{thmcov} is unknown, we select more practical regularization parameters $\beta=\delta^{1/2}$, $\beta=\delta$, $\beta=\tau\delta^{1/2}$ for QBVM, MQBVM($\alpha=0$), and our proposed PQBVM($\alpha=\alpha_*$), respectively.
  After solving the discretized full linear system, we   obtain the approximate source   $f_h$ and then compute its discrete $L_2(\Omega)$ norm error as
  $
  e_h=\|f_h-f(\cdot)\|_2.
  $
  For a fixed mesh, we would expect $e_h$ to decrease  as the noise level $\delta$ gets smaller, but the discretization errors also affect the overall accuracy, especially for our PQBVM (with $\beta=\tau\delta^{1/2}$).
  The convergence rate also depends on the regularity of $f$,
  where a smooth $f$ shows faster convergence rate than a non-smooth $f$.
  \subsection{1D and 2D examples}
  \textbf{Example 1.} Choose $\Omega=(0,\pi), T=1$, $\phi(x)=0$, and a smooth source function $$f(x)=x(\pi-x)\sin(4x).$$
  Table \ref{T1A} reports the error results and CPU times with three different regularization methods,  where the CPU times of both  MQBVM and PQBVM with the PinT direct solver are significantly faster than that of the QBVM based on MATLAB's backslash direct solver.
  With a very smooth $f$, the MQBVM in \cite{Wei_2014b} indeed shows slightly faster convergence rate than both QBVM and PQBVM.
  As shown in Figure \ref{FigEx1}, the QBVM displays undesirable artificial  oscillations for large noise levels, which was not visible in both  MQBVM and PQBVM, mainly due to the introduced Laplacian regularization term $\Delta f$ that smooths out the reconstructed approximation.

  \begin{table}[htp!]
  	\centering
  	\caption{{Error and CPU results for Ex. 1 with different mesh sizes and noise levels.}}
  	\begin{tabular}{|c|c||cccc||cccc|}\hline
  		& &\multicolumn{4}{|c|}{Errors in $L_2$ norm}&\multicolumn{4}{|c|}{CPU (in seconds)} \\
  		\hline
  		Method	&$(m,n)$$\backslash$   $\epsilon$&	 $10^{-1}$&	$10^{-2}$&$10^{-3}$&$10^{-4}$  &	 $10^{-1}$&	$10^{-2}$&$10^{-3}$&$10^{-4}$      \\ \hline

  		\multirow{3}{*}{\parbox{2cm}{QBVM\\($\beta=\delta^{1/2}$)}}
  	&(256, 256)	 &1.43e+00 	 &8.08e-01 	 &3.43e-01 	 &1.31e-01 	   &0.5 	 &0.5 	 &0.5 	 &0.5 	\\
  	&(512, 512)	 &1.41e+00 	 &8.09e-01 	 &3.57e-01 	 &1.26e-01 	   &2.6 	 &2.5 	 &2.6 	 &2.7 	\\
  	&(1024,1024)	 &1.42e+00 	 &7.97e-01 	 &3.56e-01 	 &1.28e-01 &18.6 	 &18.7 	 &18.5 	 &18.6 	\\
  		\hline \hline
  		\multirow{3}{*}{\parbox{2cm}{MQBVM\\($\beta=\delta$)}}
  &(256, 256)	 &1.66e+00 	 &6.16e-01 	 &1.12e-01 	 &1.78e-02 	   &0.1 	 &0.1 	 &0.1 	 &0.1 	\\
  &(512, 512)	 &1.66e+00 	 &6.22e-01 	 &1.16e-01 	 &1.75e-02 	   &0.3 	 &0.3 	 &0.3 	 &0.3 	\\
  &(1024,1024)	 &1.65e+00 	 &6.03e-01 	 &1.11e-01 	 &1.76e-02 &1.3 	 &1.3 	 &1.2 	 &1.2 	\\	
  		\hline \hline
  		\multirow{3}{*}{\parbox{2cm}{PQBVM\\($\beta=\tau\delta^{1/2}$)}}
 &(256, 256)	 &1.48e+00 	 &8.94e-01 	 &4.70e-01 	 &2.67e-01      &0.1 	 &0.1 	 &0.1 	 &0.1 	\\
 &(512, 512)	 &1.44e+00 	 &8.54e-01 	 &4.08e-01 	 &2.00e-01 	    &0.3 	 &0.3 	 &0.3 	 &0.3 	\\
 &(1024,1024)	 &1.42e+00 	 &8.32e-01 	 &3.83e-01 	 &1.65e-01  &1.2 	 &1.2 	 &1.2 	 &1.2 	\\  \hline
  	\end{tabular}
  	\label{T1A}
  \end{table}
 \begin{figure}[htp!]
 	\begin{center}
 		\includegraphics[width=1\textwidth]{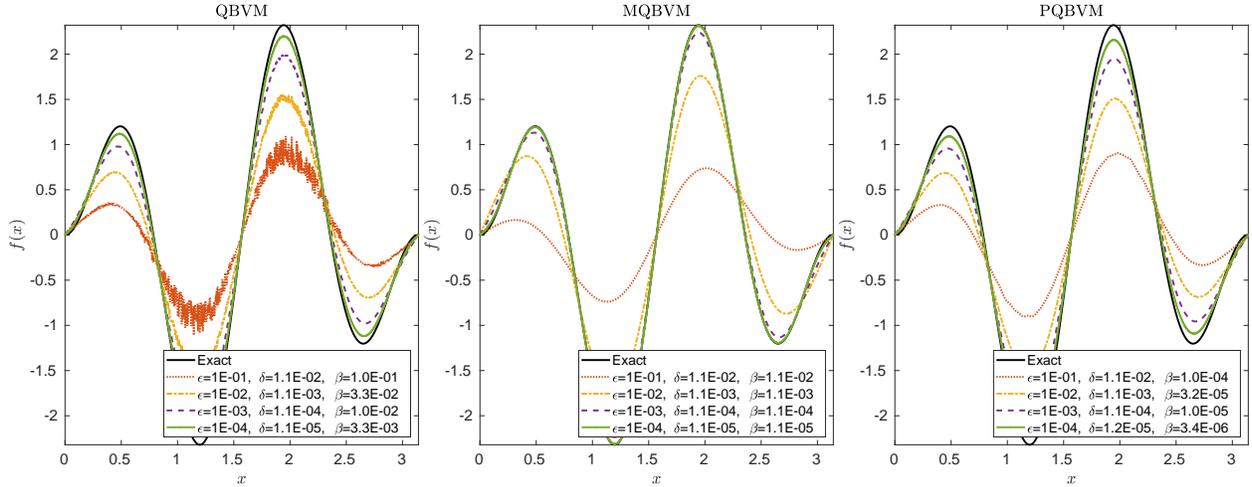}
 	\end{center}
 	\vspace{-2em}
 	\caption{Reconstructed $f(x)$ in Ex. 1 with different methods and noise levels $\epsilon\in\{10^{-1},10^{-2},10^{-3},10^{-4}\}$ (using the mesh $h=\pi/1024,\tau=T/1024$, $\alpha=\delta^{1/2}$ for QBVM, $\alpha=\delta$ for MQBVM,  and $\alpha=\tau\delta^{1/2}$ for PQBVM). The black solid curve is the exact solution.}
 	
 	\label{FigEx1}
 \end{figure}

 \textbf{Example 2.} Choose $\Omega=(0,\pi), T=1$, $\phi(x)=0$, and a non-smooth source function
 \[f(x) =\begin{cases}
 	2x, & 0\le x\le \pi/2, \\
 	2(\pi-x), & \pi/2\le x\le \pi, \\
 \end{cases}
 \]
  Table \ref{T2A} reports the error results and CPU times with three different regularization methods,  where again the CPU times of both  MQBVM and PQBVM based on our PinT direct solver are much faster
  than that of QBVM.   Figure \ref{FigEx2} illustrates the reconstructed $f(x)$ with different noise levels, where the MQBVM shows only slightly better accuracy with a non-differentiable $f$.

 \begin{table}[htp!]
 	\centering
 	\caption{{Error and CPU results for Ex. 2 with different mesh sizes and noise levels.}}
 	\begin{tabular}{|c|c||cccc||cccc|}\hline
 		& &\multicolumn{4}{|c|}{Errors in $L_2$ norm}&\multicolumn{4}{|c|}{CPU (in seconds)} \\
 		\hline
 		Method	&$(m,n)$$\backslash$   $\epsilon$&	 $10^{-1}$&	$10^{-2}$&$10^{-3}$&$10^{-4}$  &	 $10^{-1}$&	$10^{-2}$&$10^{-3}$&$10^{-4}$      \\ \hline

 		\multirow{3}{*}{\parbox{2cm}{QBVM\\($\beta=\delta^{1/2}$)}}
 	&(256, 256)	 &1.21e+00 	 &5.17e-01 	 &2.05e-01 	 &7.62e-02 	   &0.5 	 &0.5 	 &0.5 	 &0.5 	\\
 	&(512, 512)	 &1.19e+00 	 &5.18e-01 	 &2.06e-01 	 &7.86e-02 	   &2.7 	 &2.6 	 &2.6 	 &2.6 	\\
 	&(1024,1024)	 &1.21e+00 	 &5.26e-01 	 &2.04e-01 	 &7.93e-02 &18.4 	 &18.8 	 &18.4 	 &18.8 	\\	
 		\hline \hline
 		\multirow{3}{*}{\parbox{2cm}{MQBVM\\($\beta=\delta$)}}
 &(256, 256)	 &5.96e-01 	 &2.31e-01 	 &9.53e-02 	 &4.05e-02 	    &0.1 	 &0.1 	 &0.1 	 &0.1 	\\
 &(512, 512)	 &5.98e-01 	 &2.35e-01 	 &9.52e-02 	 &3.97e-02 	    &0.3 	 &0.3 	 &0.3 	 &0.3 	\\
 &(1024,1024)	 &6.21e-01 	 &2.34e-01 	 &9.53e-02 	 &3.97e-02  &1.3 	 &1.3 	 &1.3 	 &1.3 	\\	
 		\hline \hline
 		\multirow{3}{*}{\parbox{2cm}{PQBVM\\($\beta=\tau\delta^{1/2}$)}}
&(256, 256)	 &1.17e+00 	 &5.26e-01 	 &2.15e-01 	 &1.01e-01 	   &0.1 	 &0.1 	 &0.1 	 &0.1 	\\
&(512, 512)	 &1.16e+00 	 &5.16e-01 	 &2.15e-01 	 &8.95e-02 	   &0.3 	 &0.3 	 &0.3 	 &0.3 	\\
&(1024,1024)	 &1.16e+00 	 &5.11e-01 	 &2.08e-01 	 &8.50e-02 &1.3 	 &1.2 	 &1.2 	 &1.2 	\\  \hline
 	\end{tabular}
 	\label{T2A}
 \end{table}
\begin{figure}[htp!]
	\begin{center}
		\includegraphics[width=1\textwidth]{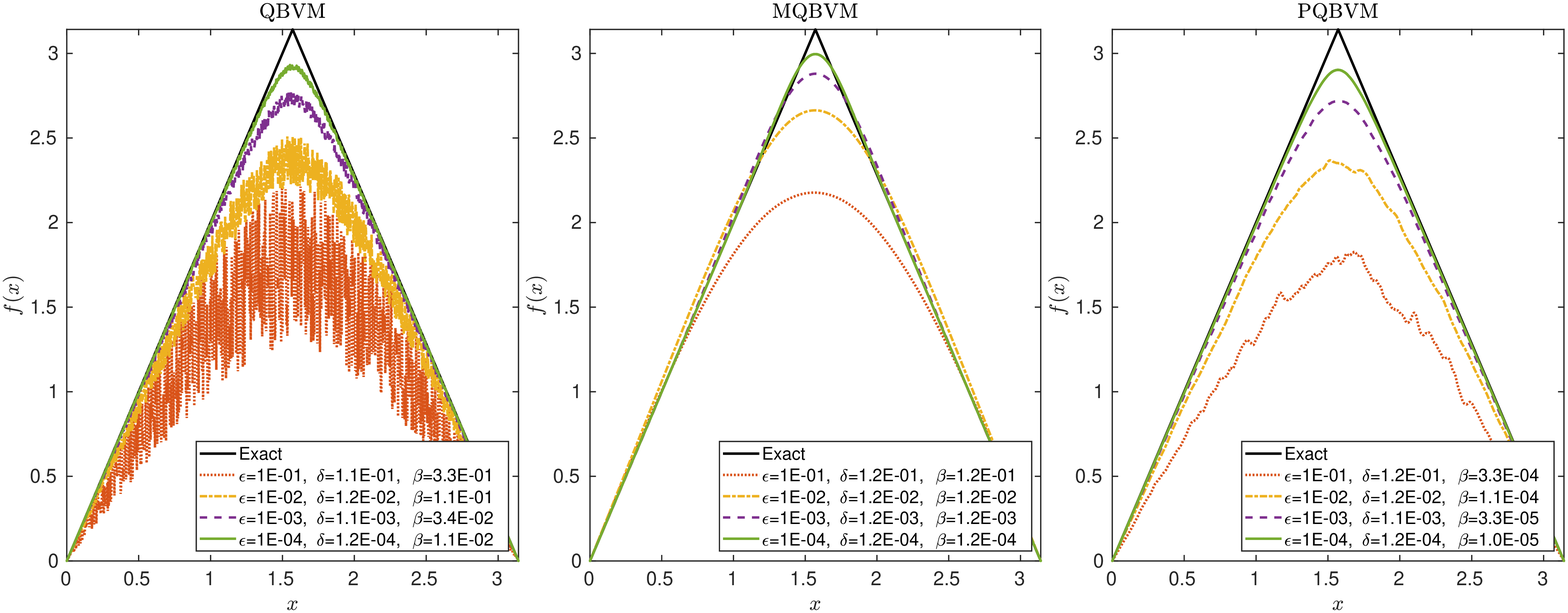}
	\end{center}
	\vspace{-2em}
	\caption{Reconstructed $f(x)$ in Ex. 2 with different methods and noise levels $\epsilon\in\{10^{-1},10^{-2},10^{-3},10^{-4}\}$ (using the mesh $h=\pi/1024,\tau=T/1024$, $\alpha=\delta^{1/2}$ for QBVM, $\alpha=\delta$ for MQBVM,  and $\alpha=\tau\delta^{1/2}$ for PQBVM). The black solid curve is the exact solution.}
	
	\label{FigEx2}
\end{figure}

\begin{table}[htp!]
	\centering
	\caption{{Error and CPU results for Ex. 3 with different mesh sizes and noise levels.}}
	\begin{tabular}{|c|c||cccc||cccc|}\hline
		& &\multicolumn{4}{|c|}{Errors in $L_2$ norm}&\multicolumn{4}{|c|}{CPU (in seconds)} \\
		\hline
		Method	&$(m,n)$$\backslash$   $\epsilon$&	 $10^{-1}$&	$10^{-2}$&$10^{-3}$&$10^{-4}$  &	 $10^{-1}$&	$10^{-2}$&$10^{-3}$&$10^{-4}$      \\ \hline

		\multirow{3}{*}{\parbox{2cm}{QBVM\\($\beta=\delta^{1/2}$)}}
		&(256, 256)	 &5.04e-01 	 &3.53e-01 	 &2.55e-01 	 &1.91e-01 	   &0.5 	 &0.5 	 &0.5 	 &0.5 	\\
		&(512, 512)	 &5.21e-01 	 &3.58e-01 	 &2.58e-01 	 &1.91e-01 	   &2.6 	 &2.5 	 &2.6 	 &2.5 	\\
		&(1024,1024)	 &5.25e-01 	 &3.57e-01 	 &2.57e-01 	 &1.92e-01 &18.4 	 &18.2 	 &18.3 	 &18.4 	\\	
		\hline \hline
		\multirow{3}{*}{\parbox{2cm}{MQBVM\\($\beta=\delta$)}}
		&(256, 256)	 &5.22e-01 	 &3.42e-01 	 &2.64e-01 	 &1.97e-01 	   &0.1 	 &0.1 	 &0.1 	 &0.1 	\\
		&(512, 512)	 &5.18e-01 	 &3.40e-01 	 &2.65e-01 	 &1.97e-01 	   &0.3 	 &0.3 	 &0.3 	 &0.3 	\\
		&(1024,1024)	 &5.16e-01 	 &3.42e-01 	 &2.63e-01 	 &1.97e-01 &1.3 	 &1.3 	 &1.3 	 &1.3 	\\ 	
		\hline \hline
		\multirow{3}{*}{\parbox{2cm}{PQBVM\\($\beta=\tau\delta^{1/2}$)}}
		&(256, 256)	 &5.15e-01 	 &3.68e-01 	 &2.84e-01 	 &2.34e-01 	   &0.1 	 &0.1 	 &0.1 	 &0.1 	\\
		&(512, 512)	 &4.97e-01 	 &3.61e-01 	 &2.73e-01 	 &2.19e-01 	   &0.3 	 &0.3 	 &0.3 	 &0.3 	\\
		&(1024,1024)	 &4.97e-01 	 &3.53e-01 	 &2.65e-01 	 &2.09e-01 &1.2 	 &1.2 	 &1.3 	 &1.2 	\\	  \hline
	\end{tabular}
	\label{T3A}
\end{table}

\begin{figure}[htp!]
	\begin{center}
		\includegraphics[width=1\textwidth]{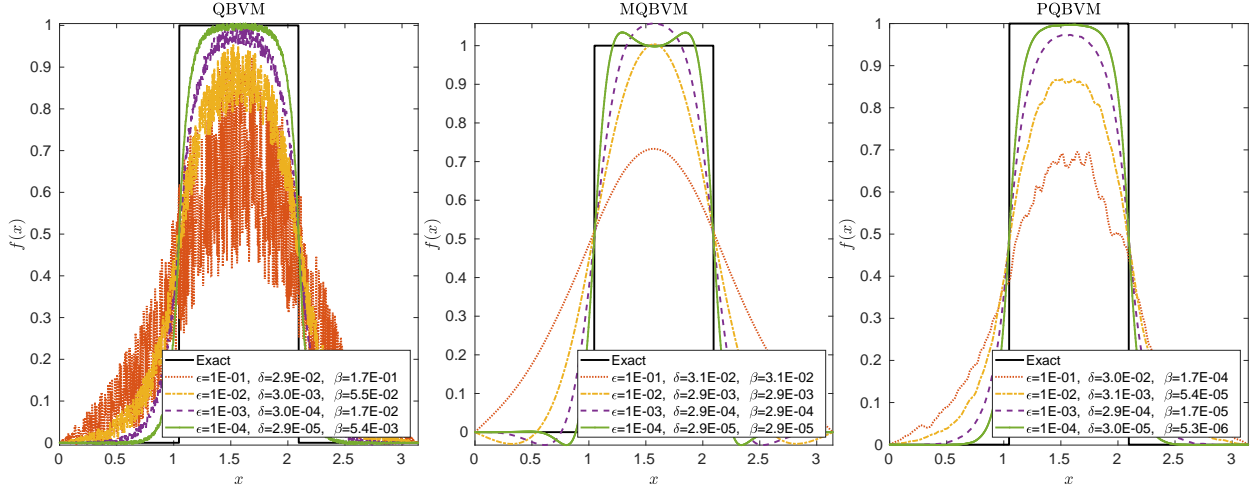}
	\end{center}
	\vspace{-2em}
	\caption{Reconstructed $f(x)$ in Ex. 3 with different methods and noise levels $\epsilon\in\{10^{-1},10^{-2},10^{-3},10^{-4}\}$ (using the mesh $h=\pi/1024,\tau=T/1024$, $\alpha=\delta^{1/2}$ for QBVM, $\alpha=\delta$ for MQBVM,  and $\alpha=\tau\delta^{1/2}$ for PQBVM). The black solid curve is the exact solution.}
	
	\label{FigEx3}
\end{figure}

 \textbf{Example 3.} Choose $\Omega=(0,\pi), T=1$, $\phi(x)=0$, and a discontinuous source function
 \[f(x) =\begin{cases}
 	1, & \pi/3\le x\le 2\pi/3, \\
 	0, & \mathrm{else}, \\
 \end{cases}
 \]
  Table \ref{T3A} reports the error results and CPU times with three different regularization methods,  where the errors of all three methods are comparable but  the CPU times of both  MQBVM and PQBVM based on our PinT direct solver are much faster.   Figure \ref{FigEx3} illustrates the reconstructed $f(x)$ with different noise levels, where the MQBVM shows more clear Gibbs phenomenon due to discontinuity and the PQBVM seems to provide most stable approximation in the sense of avoiding oscillations and overshooting near the discontinuities.

 \begin{table}[htp!]
 	\centering
 	\caption{{Error and CPU results for 2D Ex. 4 with different mesh sizes and noise levels.}}
 	\begin{tabular}{|c|c||ccc||ccc|}\hline
 		& &\multicolumn{3}{|c|}{Errors in $L_2$ norm}&\multicolumn{3}{|c|}{CPU (in seconds)} \\
 		\hline
 		Method	&$(m,n)$$\backslash$   $\epsilon$&	 $10^{-1}$&	 $10^{-2}$&$10^{-3}$  &	 $10^{-1}$&	$10^{-2}$& $10^{-3}$      \\ \hline

 		\multirow{3}{*}{\parbox{2cm}{QBVM\\($\beta=\delta^{1/2}$)}}
 		%&($ 16^2$,  16)	 &2.25e+00 	 &1.19e+00 	 &4.72e-01 	 &0.06 	 &0.06 	 &0.06 	\\
 		&($ 32^2$,  32)	 &2.21e+00 	 &1.18e+00 	 &4.98e-01 	 &2.00 	 &2.09 	 &2.08 	\\
 		&($ 64^2$,  64)	 &2.21e+00 	 &1.18e+00 	 &4.93e-01 	 &132.01 	 &127.05 	 &132.84 	\\
 		&($ 128^2$,  128)	 & --&--&--&--&--&--\\
 		\hline \hline
 		\multirow{3}{*}{\parbox{2cm}{MQBVM\\($\beta=\delta$)}}
 		&($ 32^2$,  32)	 &3.26e+00 	 &2.20e+00 	 &9.83e-01 	&0.01 	 &0.01 	 &0.01 	\\
 		&($ 64^2$,  64)	 &3.26e+00 	 &2.21e+00 	 &9.82e-01 	&0.04 	 &0.04 	 &0.05 	\\
 		&($128^2$, 128)	 &3.26e+00 	 &2.20e+00 	 &9.85e-01 	&0.28 	 &0.29 	 &0.28 	\\
 		&($256^2$, 256)	 &3.26e+00 	 &2.20e+00 	 &9.84e-01 	&2.83 	 &2.90 	 &2.91 	\\
 		&($512^2$, 512)	 &3.26e+00 	 &2.20e+00 	 &9.84e-01 	&33.68 	 &33.80 	 &33.56 	\\
 		\hline \hline
 		\multirow{3}{*}{\parbox{2cm}{PQBVM\\($\beta=\tau\delta^{1/2}$)}}
 		%&($ 16^2$,  16)	 &2.77e+00 	 &2.06e+00 	 &1.62e+00 	&0.12 	 &0.11 	 &0.11 	\\
 		&($ 32^2$,  32)	 &2.53e+00 	 &1.72e+00 	 &1.19e+00 	&0.01 	 &0.01 	 &0.01 	\\
 		&($ 64^2$,  64)	 &2.39e+00 	 &1.49e+00 	 &8.92e-01 	&0.04 	 &0.04 	 &0.04 	\\
 		&($128^2$, 128)	 &2.31e+00 	 &1.35e+00 	 &7.05e-01 	&0.29 	 &0.28 	 &0.28 	\\
 		&($256^2$, 256)	 &2.26e+00 	 &1.27e+00 	 &6.04e-01 	&3.41 	 &2.92 	 &3.44 	\\
 		&($512^2$, 512)	 &2.23e+00 	 &1.23e+00 	 &5.49e-01 	&34.42 	 &33.96 	 &33.72 	\\
 		\hline
 	\end{tabular}
 	\label{T4A}
 \end{table}

 \begin{figure}[htp!]
 	\begin{center}
 		\includegraphics[width=1\textwidth]{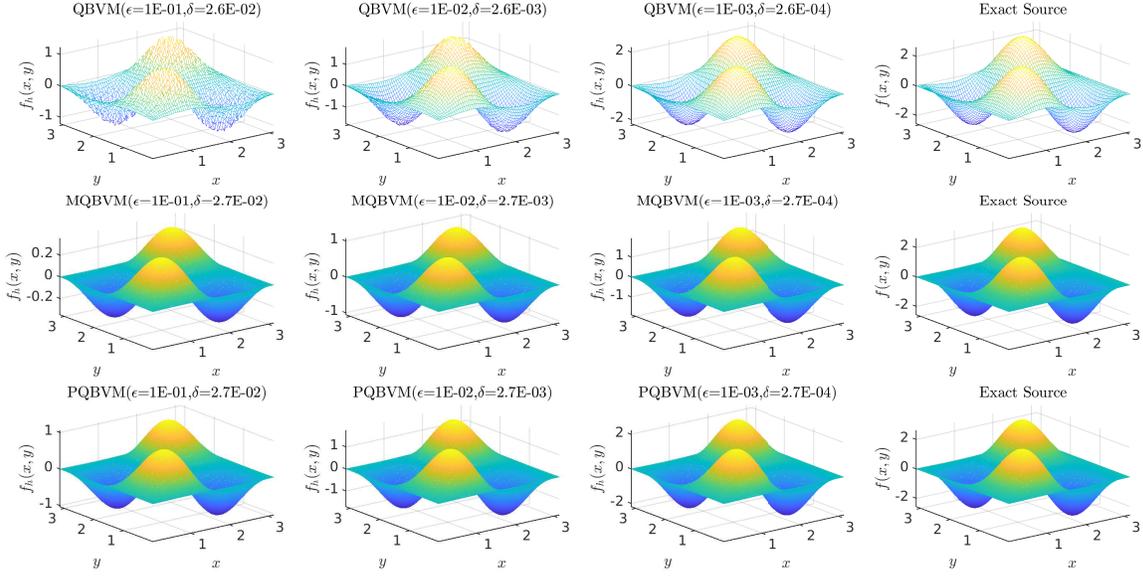}
 	\end{center}
 	\vspace{-2em}
 	\caption{Reconstructed $f(x)$ in Ex. 4 with different methods and noise levels $\epsilon\in\{10^{-1},10^{-2},10^{-3}\}$ (using the mesh $h=\pi/512,\tau=T/512$, $\alpha=\delta^{1/2}$ for QBVM, $\alpha=\delta$ for MQBVM,  and $\alpha=\tau\delta^{1/2}$ for PQBVM).}
 	
 	\label{FigEx4}
 \end{figure}

   \textbf{Example 4.} Choose $\Omega=(0,\pi)^2, T=1$, $\phi(x,y)=0$, and a smooth source function $$f(x,y)=x(\pi-x)\sin(2x)y(\pi-y)\cos(y).$$
   Table \ref{T4A} reports the error results and CPU times with three different regularization methods,  where the CPU times of both  MQBVM and PQBVM based on our PinT direct solver are much faster
   although the errors of QBVM are slightly smaller than MQBVM and PQBVM.  Notice that even for a small mesh size $(m,n)=(64^2,64)$  the CPU times are decreased from over 2 mins to about $0.04$ second, let alone a larger mesh size (such as $(m,n)=(128^2,128)$ with about 2.1 million unknowns). Here we used ``--'' to indicate the computation takes an excessively long time for MATLAB's backslash sparse direct solver.
    Figure \ref{FigEx4} illustrates the reconstructed $f(x)$ with different noise levels, where the differences between three methods are not clearly visible.
   This example demonstrates the superior computational efficiency of our proposed PinT direct solver in treating more practical 2D/3D problems that are costly to solve by the sparse direct solver.

  \subsection{Application to separable space and time-dependent source term}
  Consider the following   model \cite{Ke2020} with a given positive time-dependent source term $q(t)>0$:
   \eq \label{state_qbvm_new_t}
  \left\{\begin{array}{ll}
  	u_t -\Delta u =f(x)q(t),\ \qquad &\tn{in} \Omega\times (0,T),  \\
  	u(\cdot,t)=0, &\tn{on} \partial\Omega\times (0,T), \\
  	u(\cdot,0)=\phi , &\tn{in} \Omega, \\
  	u(\cdot,T)+{\beta(\alpha f(\cdot)-\Delta f(\cdot)})=g_\delta , &\tn{in} \Omega.
  \end{array}\right.
  \ee
  With the same finite-difference discretization, we get a  linear system of Kronecker product form
  \begin{align} \label{Akron}
  	A_h= B_q \otimes I_h - I_t\otimes \Delta_h
  \end{align}
  where the time discretization matrix $B_q$ is given by (let $q_j=q(t_j)$)
  \begin{align}
  	B_q &=\bmt
  	\alpha& 0&0 &\cdots &0 &1/\beta\\
  	-q_1 &  1/\tau &0 & \cdots &0 &0\\
  	-q_2&-1/\tau &  1/\tau & 0 & \cdots &0\\
  	\vdots&0&\ddots &\ddots &\ddots &0\\
  	-q_{n-1}& 0&\cdots &-1/\tau& 1/\tau& 0\\
  	-q_n&0&\cdots & 0 & -1/\tau &1/\tau
  	\emt \in\IR^{(n+1)\times (n+1)}.
  \end{align}	
Hence, our proposed direct PinT solver can still be applied if assuming $B_q=V_qD_q V_q^{-1}$ is diagonalizable and $V_q$ is somewhat well-conditioned. In this case, the diagonalizability of $B_q$ and   the estimate of $\kappa(V_q)$ are much more complicated to discuss as we did for the $B$ with $q(t)\equiv 1$, which will be left as future work.
The following example shows numerically it indeed works very well.
%  \hl{Xiangsheng may add a short section to  a transformation to eliminate the function $q(t)$ so that we can only diagonalize $B$. }

   \textbf{Example 5.} Choose $\Omega=(0,\pi), T=1$, $\phi(x)=0$, and the smooth source functions $$f(x)=x(\pi-x)\sin(4x), \qquad g(t)=e^{-t}+\ln(t+1)+t^2.$$
  Table \ref{T5A} reports the error results and CPU times with three different regularization methods as before and  Figure \ref{FigEx5} compares the reconstructed $f$,
  where similar conclusions can be made as in the previous Example 1.  The extra non-constant $q(t)$ term does not seems to affect the effectiveness of our proposed method, although our current analysis does not fully support this case yet.
  \begin{table}[htp!]
  	\centering
  	\caption{{Error and CPU results for Ex. 5 with different mesh sizes and noise levels.}}
  	\begin{tabular}{|c|c||cccc||cccc|}\hline
  		& &\multicolumn{4}{|c|}{Errors in $L_2$ norm}&\multicolumn{4}{|c|}{CPU (in seconds)} \\
  		\hline
  		Method	&$(m,n)$$\backslash$   $\epsilon$&	 $10^{-1}$&	$10^{-2}$&$10^{-3}$&$10^{-4}$  &	 $10^{-1}$&	$10^{-2}$&$10^{-3}$&$10^{-4}$      \\ \hline
  	\multirow{3}{*}{\parbox{2cm}{QBVM\\($\beta=\delta^{1/2}$)}}
 	&(256, 256)	   &1.23e+00 	 &6.27e-01 	 &2.60e-01 	 &8.50e-02 	     &0.5 	 &0.5 	 &0.5 	 &0.5 	\\
&(512, 512)	   &1.24e+00 	 &6.32e-01 	 &2.58e-01 	 &9.03e-02 	     &2.5 	 &2.5 	 &2.5 	 &2.5 	\\
&(1024,1024)   	 &1.23e+00 	 &6.39e-01 	 &2.54e-01 	 &9.13e-02  	 &18.3 	 &18.1 	 &18.0 	 &18.2 	\\
  	\hline \hline
  	\multirow{3}{*}{\parbox{2cm}{MQBVM\\($\beta=\delta$)}}
  	&(256, 256)	  &1.61e+00 	 &5.29e-01 	 &1.07e-01 	 &1.44e-02 	    &0.1 	 &0.1 	 &0.1 	 &0.1 	\\
  	&(512, 512)	  &1.62e+00 	 &5.73e-01 	 &1.02e-01 	 &1.56e-02 	    &0.3 	 &0.3 	 &0.3 	 &0.3 	\\
  	&(1024,1024)  	 &1.62e+00 	 &5.73e-01 	 &1.03e-01 	 &1.58e-02 	&1.3 	 &1.3 	 &1.3 	 &1.3 	\\
  	\hline \hline
  	\multirow{3}{*}{\parbox{2cm}{PQBVM\\($\beta=\tau\delta^{1/2}$)}}
  	&(256, 256)	  &1.28e+00 	 &6.94e-01 	 &3.27e-01 	 &1.61e-01 	    &0.1 	 &0.1 	 &0.1 	 &0.1 	\\
  	&(512, 512)	  &1.25e+00 	 &6.62e-01 	 &2.92e-01 	 &1.27e-01 	    &0.3 	 &0.3 	 &0.3 	 &0.3 	\\
  	&(1024,1024)  	 &1.25e+00 	 &6.53e-01 	 &2.79e-01 	 &1.10e-01 	&1.3 	 &1.2 	 &1.4 	 &1.3 	\\
  	\hline
  	\end{tabular}
  	\label{T5A}
  \end{table}
  \begin{figure}[htp!]
  	\begin{center}
  		\includegraphics[width=1\textwidth]{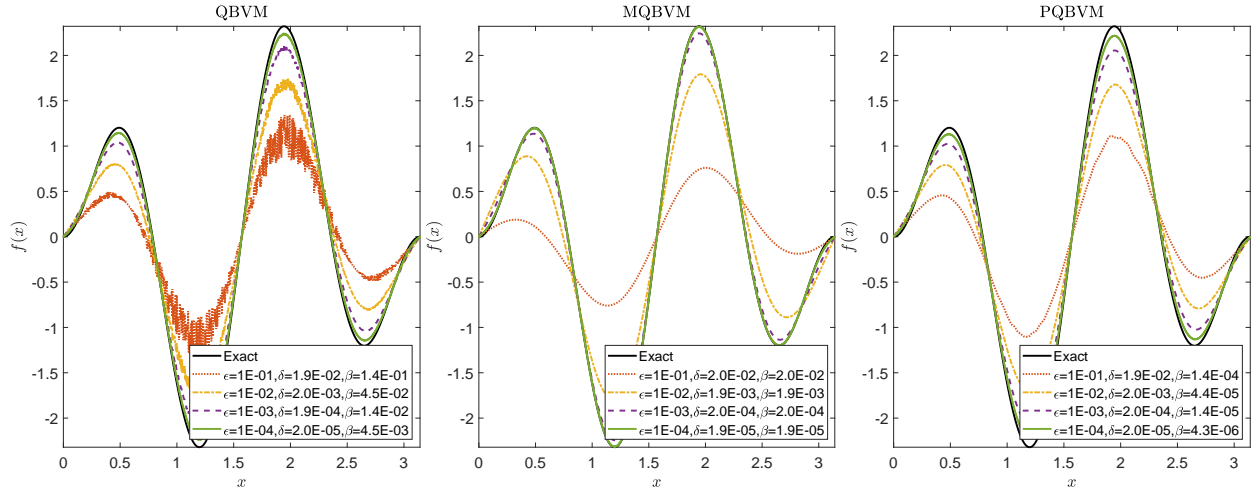}
  	\end{center}
  	\vspace{-2em}
  	\caption{Reconstructed $f(x)$ in Ex. 5 with different methods and noise levels $\epsilon\in\{10^{-1},10^{-2},10^{-3},10^{-4}\}$ (using the mesh $h=\pi/1024,\tau=T/1024$, $\alpha=\delta^{1/2}$ for QBVM, $\alpha=\delta$ for MQBVM,  and $\alpha=\tau\delta^{1/2}$ for PQBVM). The black solid curve is the exact solution.}
  	\label{FigEx5}
  \end{figure}
\section{Conclusions}  \label{secFinal}
  Inverse source problems are ill-posed and effective regularization is required for their stable numerical computation.
 The quasi-boundary value method and its variants are often used for regularizing such problems,
 which lead to large-scale ill-conditioned nonsymmetric sparse linear systems upon suitable space-time finite difference discretization.
 Such nonsymmetric all-at-once linear systems are costly to solve by either direct or iterative methods.
 In this paper we propose to modify the existing quasi-boundary value methods such that the full discretized system matrix admits a block Kronecker sum structure that can be solved by a fast diagonalization-based PinT direct solver.
 To control the roundoff errors of such a PinT direct solver,  we carefully estimate the condition number of the eigenvector matrix of the time discretization matrix, where the free parameter $\alpha=\alpha_*$ is determined for this purpose.
 Convergence analysis (with a priori choice of regularization parameter $\beta$) for our proposed parameterized quasi-boundary value method (PQBVM) is given under the special choice of $\alpha=\alpha_*$.
 Both 1D and 2D examples show our proposed PinT methods  can achieve a comparable accuracy with significantly faster CPU times.
 It is interesting to generalize our idea of integrating regularization and fast solvers to other related inverse PDE problems, such as to simultaneously recover the source term and initial value \cite{johansson2008procedure,wang2014regularized,Zheng2014,wang2019simultaneous,Wang2020}.

{\small
\bibliographystyle{siam}
\bibliography{DirectPinT,inversePDE,waveControl,ISPpint}
}

\end{document}